\documentclass[11pt,a4paper]{article}

\usepackage{amsfonts}

\usepackage{amsmath}

\usepackage[utf8]{inputenc}

\usepackage{amsthm}

\usepackage{amssymb}

\newtheorem{proposition}{Proposition}

\newtheorem*{remark}{Remark}

\newtheorem{theorem}{Theorem}

\newtheorem{corollary}{Corollary}

\newtheorem{lemma}{Lemma}

\newtheorem*{theoremA}{Theorem A}

\newtheorem*{theoremB}{Theorem B}

\newtheorem*{theoremC}{Theorem C}

\newcommand{\<}{\left\langle}
\newcommand{\rr}{\right\rangle}

\newcommand{\ode}{\overline\nabla}

\newcommand{\ws}{\widetilde{s}}
\newcommand{\wgamma}{\widetilde{\gamma}(\cdot,\:\widetilde{t})}

\begin{document}

\title{The curve shortening flow with density of a spherical curve in codimension two}

\author{ Francisco Viñado-Lereu
\thanks{
ORCID ID: 0000-0002-5828-6756. \newline Research partially supported by Universitat Jaume I research proyect UJI-B2018-35 and by MINECO research proyect MTM2017-84851-C2-2-P. The author has been supported by a postdoctoral grant from \textit{Plan de promoción de la investigación de la Universitat Jaume I del año 2018 Acción 3.2.} POSDOC-A/2018/32 - grupo 041.
}
}

\date{}

\maketitle

\vspace{-1cm}
\begin{abstract}
In the present paper we carry out a systematic study about the flow of a spherical curve by the mean curvature flow with density in a  3-dimensional rotationally symmetric space with density $(M^3_w,\:g_w,\:\xi)$ where the density $\xi$ decomposes as sum of a radial part $\varphi$ and an angular part $\psi$. We analyse how either the parabolicity or the hyperbolicity of $(M^3_w,\:g_w)$ condition the behaviour of the flow when the solution goes to infinity.
\end{abstract}



\section{Introduction}

A n+1-dimensional \textit{manifold with density} $(M,\:g,\:\xi)$ is a Riemannian manifold $(M,\:g)$ and a function $\xi:M\rightarrow \mathbb{R}$. In this type of manifold we may calculate the \textit{weight volume} or \textit{volume with density} of the k-dimensional immersed submanifolds $\iota:P^k\rightarrow M^{n+1}$ as:
\begin{align}
V_\xi(P):=\int_P\:e^{\xi\circ\iota}\:dv_{g_P},
\end{align}
where $g_P\equiv\iota^\star g$ is the induced metric over the manifold $P$ by the immersion $\iota$. We shall denote by $dv_{\xi,\:P}$ or $dv_{\xi}$ to $ e^\xi\:dv_{g_P}$ the volume element associated to a density.

In this context we have a natural generalization to the mean curvature vector of a submanifold as the negative $L^2$-gradient of the k-dimensional functional of volume with density. We shall call to this vector field \textit{mean curvature vector with density} and it shall be denoted by $\vec{H}_\xi$. It has the form:
\begin{align}
\vec{H}_\xi:= \vec{H}-\big(\nabla^M\xi\big)^{\perp},
\end{align}
where $\vec{H}$ is the mean curvature vector of the submanifold and $\big(\nabla^M\xi\big)^{\perp}$ is the orthogonal projection of $\nabla^M\xi$, gradient of $\xi$ in $(M,\:g_M)$, to the submanifold. In the particular case where $k=n=1$, which is the case of curves on a surface, we shall change $\vec{H}$ by $\vec{k}$ the geodesic curvature vector and we shall denote by $\vec{k}_\xi$ the new vector field, we shall call this vector field the \textit{geodesic curvature vector with density}.

This fact motivate us to study the following flow:
\begin{align}
\left\lbrace \begin{array}{ccc}
\dfrac{\partial F}{\partial t}(p,\:t)& = &\Big( \vec{H}_\xi\Big)_{F(p,\:t)}, \\
F(p,\:0)& =&F_0(p),\end{array} \right.
\end{align}
where $F_0:P^k\rightarrow M^{n+1}$ is a k-dimensional immersed submanifold, as analogous flow to the mean curvature flow in the context of the geometry with density. This flow is called \textit{the mean curve flow with density} ($\xi$MCF for short). In the particular case where $k=1$, case of curves, this problem is also called \textit{the curve shortening problem with density}. Some works performed in this context are \cite{ScSm02,BoMi10,MiVi16,MiVi18}; let us remark that these authors did not necessarily use this name for the flow. Other authors had indirectly explored this problem to study the mean curvature flow of submanifolds with some symmetries \cite{Hu90,Sm96}. All these works were done for hypersurfaces ($k=n$).

Given a n-dimensional immersed submanifold $\iota:P^n\rightarrow M^{n+1}$ with $n\geq 2$ we shall define the \textit{mean curvature with density} as:
\begin{align*}
H_\xi:=H-g_{M}(\nabla^M\xi,\:N),
\end{align*}
where $N$ is a unit normal field to the hypersurface and we consider the following convention signs:
$ AX=-\nabla^M_X N,\:H=trA=\sum_{i=1}^n g_{M}(e_i,\:Ae_i), $
with $A$ as the Weingarten map and $\lbrace e _i\rbrace_{i=1}^n$ an orthonormal frame of the hypersurface. If $n=1$ we define  the \textit{geodesic curvature with density} as:
\begin{align*}
k_\xi:=k-g_{M}(\nabla^M\xi,\:N).
\end{align*}
Given an immersion such that $H_\xi=0$ or $k_\xi=0$, depending of the dimension, we shall call to this immersion \textit{$\xi$-minimal}.

In the present paper we work on the manifold with density $(M^3_w,\:g_w,\:\xi)$ that we describe below. Let $(M^3_w,\:g_w)$ be a 3-dimensional smooth rotationally symmetric space:
\begin{align}\label{DefVaAmbiRota}
M^3_w&\equiv[0,\:\infty)\times\mathbb{S}^2,
\nonumber\\
g_w&\equiv \pi^\star dr^2+(w\circ\pi)^2\sigma^\star g_{\mathbb{S}^2},
\end{align}
where $w:[0,\:\infty)\rightarrow \mathbb{R},\:w=w(r),$ is a smooth map such that $w\vert_{(0,\:\infty)}>0$ and $w(0)=0$, $g_{\mathbb{S}^2}$ is the metric over the 2-sphere with Gauss curvature equal to one and  $\pi:M^3_w\rightarrow[0,\:\infty)$, $\sigma:M^3_w\rightarrow \mathbb{S}^2$ are the natural projections. As the manifold is smooth, we may prove that the function $w$ satisfies:
\begin{align*}
w'(0)=1,
\end{align*}
also $o\equiv \lbrace 0\rbrace\times\mathbb{S}^2$ is a pole for the Riemannian manifold $(M^3_w,\:g_w)$. Regarding the density, let $\xi:M^3_w-\lbrace o\rbrace\rightarrow\mathbb{R}$ be a smooth application such that:
\begin{align}\label{ConDen}
\xi(x)=\varphi\circ\pi(x)+\psi\circ\sigma(x),
\end{align}
with $\varphi\in C^\infty((0,\:\infty))$ and $\psi\in C^\infty(\mathbb{S}^2)$. We note that $\xi$ is not defined in the pole $o$.

In this manifold we are able to study the evolution of a closed smooth embedded spherical curve by the mean curvature flow with density. We shall denote by $\mathcal{A}$ to this set of curves, that it is given exactly by
\begin{align*}
\mathcal{A}:=\lbrace \gamma:\mathbb{S}^1\rightarrow M^3_w\:\vert\:\gamma\:\text{is\:a\: smooth\:embedded\:curve\:such\:that\:}\pi(\text{Im}\:\gamma)=\lbrace r\rbrace\subseteq (0,\infty) \rbrace
\end{align*}

Actually, the aim of this paper is to make a systematic study of the following problem:

\textit{
If we consider the Riemannian manifold with density $(M^3_w,\:g_w,\:\xi)$ given by (\ref{DefVaAmbiRota}) and (\ref{ConDen}) then, in this space, we may study the following initial value problem:
\begin{align}\label{PVI}
\left\lbrace \begin{array}{ccc}
\dfrac{\partial}{\partial t}\gamma(p,\:t) & = & (\overrightarrow{H}_{\xi})_{\gamma(p,\:t)}, \\
\gamma(\cdot,\:0)& =&\gamma_0\in\mathcal{A},\end{array} \right.
\end{align}
where $(\overrightarrow{H}_{\xi})_{\gamma(\cdot,\:t)}$ denotes the mean curvature vector with density of the curve $\gamma(\cdot,\:t)$.
}

Other previous works in the field about curve shortening problem in codimension greater than or equal to two are \cite{LiDez07,Zho17}. In these two works the ambient manifold is compact and the authors assume some restrictions on the initial curve to guarantee that the maximal time of the solution is infinite. In \cite{LiDez07} the initial curve is a ramp and in \cite{Zho17} the initial curve is a graphical curve. Also in \cite{LiDez07} we can find some more general results.

Regarding the $\xi$MCF, we note that most of the results in the literature about $\xi$MCF concern radial densities in the Euclidean space \cite{ScSm02,BoMi10,MiVi16}, thus $(M^3_w,\:g_w,\:\xi)$ is a good manifold to keep increasing the understanding of this type of problems. On the other hand, the election of $\mathcal{A}$ as set of initial condition for the problem (\ref{PVI}) is motived because this family of curves is good enough to guarantee that if the solution has finite maximal time then the curve collapses to a point, moreover, the property of being embedded is preserved throughout the flow. In general, when we study the evolution of a curve in codimension greater than or equal to two, the properties above described are false. An interesting problem would be to look for other families of curves with these good properties in codimension greater than or equal to two.

It is equally important to remark that these types of flows (\ref{PVI}) had been indirectly and partially studied in the 3-dimensional Euclidean space without density ($\xi=0$). This situation was studied to understand the mean curvature flow of lagrangian spherical surfaces in $\mathbb{C}^2$ in \cite{CasLerMi18}. This fact is motivated by the link between the mean curvature flow and the $\xi$MCF, which is explicitly detalied in the article \cite{MiVi18}.

In our situation we may look the geodesic spheres 
\begin{align*}
S_r:=\lbrace p\in M^3_w\vert\: \pi(p)=r\rbrace
\end{align*} as Riemannian manifolds with density $(S_r,\:g_{S_r},\:\psi)$ with $g_{S_r}\equiv w^2(r)\: g_{\mathbb{S}^2}$ and $\psi\equiv\psi\circ\sigma$,  with a slight overuse in notation about the density. Besides, given $\gamma\in\mathcal{A}$ we denote by $\widetilde\gamma$ to the curve as curve in $S_{r}$.

With the selection of a spherical curve as initial condition, it becomes natural to our mind the following problem:

Let $\gamma_0\in\mathcal{A}$ and let $\widetilde{\gamma}:\mathbb{S}^1\times[0,\:\widetilde{T})\rightarrow S_{r_0},\:\widetilde{\gamma}=\widetilde{\gamma}(p,\:\widetilde{t}),$ be a smooth application such that:
\begin{align}\label{PVIesfera}
\left\lbrace \begin{array}{ccc}
\dfrac{\partial}{\partial\widetilde{t}}\widetilde\gamma(p,\:\widetilde{t}) & = & (\overrightarrow{k}_{S_{r_0},\:\psi})_{\widetilde{\gamma}(p,\:\widetilde{t})}, \\
\widetilde\gamma(\cdot,\:0)& =&\widetilde\gamma_0,\end{array} \right.
\end{align}
where $\overrightarrow{k}_{S_{r_0},\:\psi}$ denotes the geodesic curvature vector with density of the curve in the Riemannian manifold $(S_{r_0},\:g_{S_{r_0}},\:\psi)$.
This problem is included in the theory \cite{An90,An91,Oak94}, as well as the behaviour when the curve collapses to a point is included in the work \cite{Zhu98} and the behaviour when the solution exists for all time is in \cite{MiVi16}.

In the present work, we study in which way could we use the known results of (\ref{PVIesfera}) to understand the solution of (\ref{PVI}). On the other hand, we also note that the problem (\ref{PVI}) is a generalization of the problem (\ref{PVIesfera}).

In \cite{MiVi16} it was shown that in the study of the $\zeta$MCF of a curve in the Euclidean plane $\mathbb{R}^2$ with a radial density $\zeta=\zeta(r)$, a very relevant fact is that the sign and the zeros of the function $r\mapsto\dfrac{1}{r}+\zeta'(r)$ condition the dynamic of the solution. Analogously, in our situation the dynamic of the solution is influenced by the function:
\begin{align}
\vec{B}:M^3_w&\longrightarrow TM^3_w
\nonumber\\
p&\longmapsto \vec{B}(p):=\Big(-\dfrac{w'(\pi(p))}{w(\pi(p))}-\varphi'(\pi(p))\Big)\partial_r\vert_p,
\end{align}
whose scalar version is
\begin{align}\label{DefBEs}
B:[0,\:\infty)&\longrightarrow \mathbb{R}
\nonumber\\
r&\longmapsto B(r):=\dfrac{w'(r)}{w(r)}+\varphi'(r),
\end{align}
that is $\vec{B}(p)=-B(\pi(p))\partial_r\vert_p$. This scalar version generalizes the function given in \cite{MiVi16}. 

We note that the function (\ref{DefBEs}) has the following interpretation: Let $(\mathbb{R}^2,\:g_{w},\:\varphi)$ be the Riemannian manifold where $g_{w}:=dr^2+w^2(r)\: g_{\mathbb{S}^1}$ and $\varphi=\varphi\circ\pi$ is the density that appears in (\ref{ConDen}). Then, $B(r)$ is the geodesic curvature with density $\varphi$ of the $C_r$ circle centered at the origin whose radius is $r$, if we consider $-\partial_r$ as unit normal vector to $C_r$. Therefore, if $B(r)=0$, this means that the circle $C_r$ is $\varphi$-minimal. For this reason,  given a geodesic sphere $S_r$ such that $B(r)=0$ we shall say that it is \textit{$B$-minimal}.

Once proved the existence and unicity of solution for the problem (\ref{PVI}), we get that, if the maximal time of the solution of (\ref{PVI}) is finite, then:

\begin{theoremA}
Let $\gamma_0\in\mathcal{A}$ and let $\gamma:\mathbb{S}^1\times[0,\:T)\rightarrow M^3_w$ be the unique maximal solution for the initial value problem (\ref{PVI}) with $\gamma_0$ as initial condition, whose maximal time $T$ is finite, then, the curve collapses to a point $p\in M^3_w$ and
\begin{itemize}
\item[i)]  If $p\neq o$ the curve collapses to a spherical round point in the geodesic sphere $S_{\pi(p)}$.
\item[ii)] If $p=o$ and there exists $\widetilde{\varphi}\in C^1([0,\:\infty))$ such that $\widetilde\varphi\vert_{(0,\:\infty)}=\varphi$ then, a blow-up centered at this point gives a limit flow by the $\psi$MCF in $(S_{r_0},\:g_{S_{r_0}},\:\psi)$ that $C^\infty$-subconverges, after a reparametrization of the curves, to a closed $\psi$-minimal curve.  
\end{itemize}
\end{theoremA}

The following important goal is to study the behaviour of the flow when the solution exists for all time:

\begin{theoremB}
Let $\gamma_0\in\mathcal{A}$ and let  $\gamma:\mathbb{S}^1\times[0,\:T)\rightarrow M^3_w$ be the unique maximal solution of the initial value problem (\ref{PVI}) with $\gamma_0$ as initial condition. Then, if the solution exists for all $t$, that is, $T=\infty$, it is bounded and it does not go to the pole $o$, thus the flow $C^\infty$-subconverges, after a reparametrization of the curves $\gamma(\cdot,\:t)$, to a closed $\psi$-minimal spherical curve contained in the $B$-minimal geodesic sphere $S_{\lim_{t\rightarrow\infty} \pi(\text{Im}\:\gamma(\cdot,\:t))}$.
\end{theoremB}

In the particular case where $\widetilde{\varphi}\in C^1([0,\:\infty))$, such that $\widetilde\varphi\vert_{(0,\:\infty)}=\varphi$, exists we may eliminate the third condition: the flow does not go to the pole $o$. However, generally, if the function $\varphi$ does not have a $C^1$-extension to the pole $o$, it is possible that the flow collapses to the pole for $T=\infty$.

On the other hand, the $B$-minimal geodesic spheres are barriers for the flow. Thus, given a spherical curve between two  $B$-minimal geodesic spheres then, the flow is contained between these two spheres. Taking this into account, it is not difficult to find situations in which the solution is bounded.

We note that in the previous works \cite{LiDez07,Zho17} the ambient manifold is compact so the solution always is bounded. However, in our situation it is not compact and allows this new situation, that is,  the solution may have infinite maximal time and the curve may go to infinity. A required condition for this situation is that the function $B$ is negative at infinity.

In this last situation, the key for understanding the behaviour of the solution is given by a relation between the area of the geodesic spheres $S_r$ and the function $B$. More accurately, the behaviour of the flow in the infinity  is given by the integral:
\begin{align}
\int^\infty\dfrac{1}{B(r)\:Area(S_r)}\:dr
\end{align}
if this integral converges or diverges then, the behaviour of the flow is different. Therefore if:
\begin{align}
\int^\infty\dfrac{1}{B(r)\:Area(S_r)}\:dr\sim\int^\infty\dfrac{1}{Area(S_r)}\:dr
\end{align}
the behaviour of the flow is given by the parabolicity or hyperbolicity of the manifold $M^3_w$, as in these spaces it is characterized by the second integral \cite{Gri99}. Following this trend, the result that we have obtained in this way is:

\begin{theoremC}
Let $\gamma_0\in\mathcal{A}$ and let  $\gamma:\mathbb{S}^1\times[0,\:T)\rightarrow M^3_w$ be the unique maximal solution of the initial value problem (\ref{PVI}) with $\gamma_0$ as initial condition then, if the solution exists for all $t$, that is, $T=\infty$, and it is not bounded:
\begin{itemize}
\item[a)] If $M^3_w$ is parabolic and the $\liminf_{r\rightarrow\infty} B(r)$ is finite then, the flow topologically subconverges to $\gamma_\infty:\mathbb{S}^1\rightarrow [0,\:\infty]\times\mathbb{S}^2,\:p\mapsto (\infty,\:\chi(p))$, where $\chi:\mathbb{S}^1\rightarrow \mathbb{S}^2$ is a smooth embedded closed $\psi$-minimal curve in $(S_{r_0},\:g_{S_{r_0}},\:\psi)$.
\item[b)]  If $M^3_w$ is hyperbolic and the $\limsup_{r\rightarrow\infty}B(r)\neq 0$ then, the flow either
\begin{itemize}
\item topologically converges to $\gamma_\infty:\mathbb{S}^1\rightarrow [0,\:\infty]\times\mathbb{S}^2,\:p\mapsto (\infty,\:\widetilde{\gamma}(p,\:\widetilde{T}))$,
\item or topologically converges to a point $p_\infty\in\mathbb{S}_\infty\equiv \lbrace\infty\rbrace\times\mathbb{S}^2\subset [0,\:\infty]\times\mathbb{S}^2$ in the infinite radius sphere,
\end{itemize}
where $\widetilde\gamma$ is the solution of (\ref{PVIesfera}) with initial condition $\widetilde\gamma_0$ and $\widetilde{T}$ shall be defined in (\ref{CamTi}).
\end{itemize}
\end{theoremC}

Let us remind that to achieve the situation of the last theorem, the solution is not bounded, we need that $B\vert_{[r^\star,\:\infty)}<0$ for some $r^\star\in(0,\:\infty)$.

The case b) of the last theorem is satisfied in the Euclidean space $\mathbb{R}^3$ with a Gaussian density $\xi(x)=e^{-\mu^2r^2 /2}$. This case shall be fully studied in the section 6 expanding our understanding of the Gaussian mean curvature flow \cite{BoMi10,MiVi16}.

In the paper \cite{MiVi18} was detailed the link between the MCF and the $\xi$MCF. This relation gives us an equivalence between the flows that in our situation is the following: the $\xi$MCF $\gamma:\mathbb{S}^1\times[0,\:T)\rightarrow M^3_w$ of a curve $\gamma_0\in\mathcal{A}$ is equivalent to the MCF $F:N\times[0,\:T)\rightarrow \widehat{M}$ of a submanifold $F_0:N\rightarrow \widehat{M}$ in the (m+3)-dimensional smooth Riemannian manifold $(\widehat{M},\:\widehat{g})$ given by:
\begin{align*}
\widehat{M}&:=M^3_w\times Q=[0,\:\infty)\times\mathbb{S}^2\times Q,\\
\widehat{g}&:=\widehat\pi^\star g_w+\Big(\dfrac{e^{\xi\circ\widehat{\pi}}}{Vol_{g_Q}(Q)}\Big)^{2/m}\widehat\sigma^\star\:g_Q,
\end{align*}
where $(Q,\:g_Q)$ is a m-dimensional smooth compact Riemannian manifold and $\widehat\pi:\widehat M\rightarrow M^3_w$, $\widehat\sigma:\widehat M\rightarrow Q$ are the natural projections. And the initial submanifold $F_0$ is a (m+1)-dimensional smooth submanifold such that 
\begin{align*}
F_0:N=\mathbb{S}^1\times Q\rightarrow \widehat{M},\:F_0(\alpha,\:q):=(\gamma_0(\alpha),\:q). 
\end{align*}
Therefore, this work leads us to expand our understanding about the MCF of submanifolds in codimension two. Some works in this area are \cite{CheJi01,Wang01}.

This paper is structured in the following way. In section 2, we introduce the basic results about the curve shortening problem with density on a surface. In section 3, we give the result of existence and unicity of solution for the problem (\ref{PVI}) and we show the relation between the solutions of (\ref{PVI}) and (\ref{PVIesfera}). In section 4, we analyse the situation where the solution has finite maximal time. In section 5, we perform a study about the asymptotic behaviour of the solution when it has infinite maximal time. Finally, in section 6, we carry out a detailed study of the problem (\ref{PVI}) with the Gaussian density in the 3-dimensional Euclidean space.

\section{Preliminaries}

\subsection{Curve shortening problem with density}

Let $(\overline{M},\:\overline{g},\:\psi)$ be a 2-dimensional smooth Riemannian manifold with density and let  $\gamma_0:\mathbb{S}^1\rightarrow \overline{M}$  be a smooth curve then, we shall call the \textit{curve shortening problem with density} ($\psi$MCF) of $\gamma_0$ to the solution $\gamma:\mathbb{S}^1\times[0,\:T)\rightarrow\overline{M}$ of the problem:
\begin{align}\label{CSPD}
\left\lbrace \begin{array}{ccc}
\dfrac{\partial}{\partial t}\gamma(p,\:t) & = & \big(\overrightarrow{k}_{\psi}\big)_{\gamma(p,\:t)}, \\
\gamma(\cdot,\:0)& =&\gamma_0,\end{array} \right.
\end{align}
where $\vec{k}_\psi$ is the geodesic curvature vector with density of the curve.

In the particular case where the density is $\psi=0$, then the problem is the classic curve shortening problem widely studied \cite{Ga84,GaHa86,Gra89,Ga90b,ChoZhu01}.

As already remarked in \cite{MiVi16}, we may use the theory of S. Angenent \cite{An90,An91} to guarantee the existence and unicity of solution for the problem (\ref{CSPD}). This theory, together with the later work of Oaks \cite{Oak94}, allows us to enunciate the following theorem:

\begin{theorem}{\cite{An90,An91,Oak94} \label{TAnOaks}}
Let $\gamma_0:\mathbb{S}^1\rightarrow \overline{M}$ be a simple $C^2$ curve. Then the solution to (\ref{CSPD}) with initial condition $\gamma_0$ either collapses to a point on $\overline{M}$ in finite time or exists for infinite time.
\end{theorem}
\noindent This theorem was already written for this flow in the article \cite{MiVi16}. Moreover, we can use the work of Xi-Ping Zhu \cite{Zhu98} to obtain that:

\begin{theorem}{\cite{Zhu98} \label{Zhu}}
Let $\gamma_0:\mathbb{S}^1\rightarrow \overline{M}$ be a simple $C^2$ curve. Then, if the solution to (\ref{CSPD}) with initial condition $\gamma_0$ has finite maximal time, the solution collapses to a round point.
\end{theorem}

The case in which the solution exists for all time was studied in \cite{MiVi16} by V. Miquel and the author in a Riemannian manifold with density with the following properties in the region where the curve moves:
\begin{align}\label{MiVi16Con1}
\left. \begin{array}{l}
\vert\overline\nabla^j K\vert\leq C_j,\:\vert\overline\nabla^j\psi\vert\leq P_j,\:0<E\leq e^\psi\leq D, \\
\text{for\:some\:constants\:}C_j,\:P_j,\:E,\:D;\:j=0,\:1,\:2,\cdots\end{array} \right.
\end{align}
\begin{align}\label{MiVi16Con2}
&\text{The isoperimetric profile } \mathcal{I}\text{ is a well defined continuous function which}
\nonumber\\ &\text{satisfies } \lim_{a\rightarrow a_0}\mathcal{I}(a)=0 \text{ implies } a_0=0.
\end{align}
Where $K$ is the Gauss curvature of the surface $(\overline{M},\:\overline{g})$ and $\overline\nabla$ is the covariant derivative of $(\overline{M},\:\overline{g})$. The authors obtained the following theorem:
\begin{theorem}{\label{TViMi} \cite{MiVi16} }
Let $(\overline{M},\:\overline{g},\:\psi)$ be an orientable 2-Riemannian manifold with density satisfying (\ref{MiVi16Con2}). Let $\gamma(\cdot,\:t)$ be a solution of the $\psi$MCF (\ref{CSPD}) with initial condition an embedded curve $\gamma_0:\mathbb{S}^1\rightarrow\overline{M}^2$. If this solution exists for every $t\in[0,\:\infty)$, and $\gamma(\mathbb{S}^1,\:t)$ is contained in a fixed compact domain $U$ where the conditions (\ref{MiVi16Con1}) are satisfied, then there is a reparametrization $\widetilde{\gamma}(\cdot,\:t)$ of $\gamma(\cdot,\:t)$ such that there is a sequence $\lbrace\widetilde{\gamma}(\cdot,\:t_k)\rbrace_{k\in\mathbb{N}}$, $t_k\rightarrow\infty$, which $C^m$-converges to a closed $\psi$-minimal curve of $\overline{M}^2$ for every $m\in\mathbb{N}$.
\end{theorem}

In our particular case (\ref{PVIesfera}) the Riemannian manifold with density is $(S_{r_0},\:g_{S_{r_0}},\:\psi)$. This manifold is compact and the density is smooth thus, we have that the properties (\ref{MiVi16Con1}) and (\ref{MiVi16Con2}) are satisfied, so we may use Theorem \ref{TViMi} in our situation.

Other important properties that may be obtained from the work of S. Angenent \cite{An90,An91} are:

\begin{theorem}{\label{TeCSPPreEm}\bf (Preservation of the embedded property)} Let $\gamma_0:\mathbb{S}^1\rightarrow \overline{M}$ be a smooth embedded curve and let $\gamma:\mathbb{S}^1\times[0,\:T)\rightarrow\overline{M}$ be the solution of (\ref{CSPD}) with initial condition $\gamma_0$ then, $\gamma(\cdot,\:t):\mathbb{S}^1\rightarrow\overline{M}$ is embedded for all $t\in[0,\:T)$.
\end{theorem}

\begin{theorem}{\label{TeCSPComPrin}\bf (Comparison principle)}
Let $\gamma_1:\mathbb{S}^1\times[0,\:T_1)\rightarrow \overline{M}$,  $\gamma_2:\mathbb{S}^1\times[0,\:T_2)\rightarrow \overline{M}$ be solutions of the initial value problem (\ref{CSPD}) such that $\gamma_1(\cdot,\:0)$ and $\gamma_2(\cdot,\:0)$ are immersed curves.  If $Im\:\gamma_1(\cdot,\:0)\cap Im\:\gamma_2(\cdot,\:0)=\emptyset$ then $Im\:\gamma_1(\cdot,\:t)\cap Im\:\gamma_2(\cdot,\:t)=\emptyset$ for all $t\in[0,\:\min\lbrace T_1,\:T_2\rbrace)$.
\end{theorem}

\section{The flow}

In this section we prove the existence and unicity of the solution of the problem (\ref{PVI}). Also, we show the relation between the solutions of the problems (\ref{PVI}) and (\ref{PVIesfera}), as well as some properties of the problem (\ref{PVI}).

\begin{theorem}{\label{TeEyU}\bf (Existence and uniqueness)}
The initial value problem (\ref{PVI}) has a unique solution $\gamma:\mathbb{S}^1\times[0,\:T)\rightarrow M^3_w$ given by
\begin{align}\label{SoWar}
\gamma(p,\:t):=exp\Big(\widetilde\gamma(p,\:\widetilde{t}(t)),\:(R(t)-r_0)\partial_r\Big),
\end{align}
where $\widetilde\gamma$ is the unique solution for the initial value problem (\ref{PVIesfera}) with $\widetilde{\gamma}_0$ as initial condition, $R(t)$ is the unique solution for the ODE's system:
\begin{align}\label{SisRadio}
\left\lbrace \begin{array}{ccc}
R'(t) & = & -B(R(t)), \\
R(0) & =&r_0,\end{array} \right.
\end{align}
and
\begin{align}\label{CamTi}
\widetilde{t}:[0,\:T)&\longrightarrow [0,\:\widetilde{T})
\nonumber\\
t&\longmapsto \widetilde{t}(t):=\int_0^t\Big(\dfrac{w(r_0)}{w(R(t))}\Big)^2\:dt,
\end{align}
with $\widetilde{T}\equiv\int_0^T\Big(\dfrac{w(r_0)}{w(R(t))}\Big)^2\:dt$.
\end{theorem}
\noindent Prior to the proof, we need some previous lemmas.

Given $\gamma\in\mathcal{A}$ we shall denote by $\vec{k}_{S_{r},\:\psi}$ the geodesic curvature vector with density of the curve $\widetilde\gamma$ as curve of $(S_r,\:g_{S_r},\:\psi)$.

\begin{lemma}\label{LeExpCurMCurEs}
Given $\gamma\in\mathcal{A}$:
\begin{align}
\vec{H}_\xi=\vec{k}_{S_{r},\:\psi}+\vec{B}(\gamma)
\end{align}
where $r=\pi(Im\:\gamma)$.
\end{lemma}
\begin{proof}
Let $\lbrace\tau,\:\nu,\:\partial_r\rbrace$ be an orthonormal frame over the curve with $\tau$ the unit tangent vector to the curve, $\nu$ unit normal to the curve and tangent to the geodesic sphere where the curve is contained, we note that $\lbrace\tau,\:\nu\rbrace$ is an orthonormal frame over the curve in the geodesic sphere $(S_r,\:g_{S_r})$ . Now we may calculate the expression of the mean curvature vector with density:
\begin{align*}
\vec{H_\xi}&=\vec{H}-\big(\overline{\nabla}\xi\big)^\perp
=\<\ode_\tau \tau,\:\nu\rr\nu+\<\ode_\tau \tau,\:\partial_r\rr\partial_r-\<\ode\xi,\:\nu\rr\nu
-\<\ode\xi,\:\partial_r\rr\partial_r
\\
&=\Big(\<\ode_\tau \tau,\:\nu\rr-\<\ode\xi,\:\nu\rr\Big)\nu+\Big(\<\ode_\tau \tau,\:\partial_r\rr-\<\ode\xi,\:\partial_r\rr\Big)\partial_r
\\
&=\Big(\<\nabla^{S_{r}}_\tau \tau+\alpha^{S_{r}}(\tau,\:\tau),\:\nu\rr-\<\ode(\psi\circ\sigma),\:\nu\rr\Big)\nu
\\&\qquad+\Big(-\dfrac{\<\tau,\:\tau\rr}{w(r)}\<\ode w,\:\partial_r\rr-\<\ode(\varphi\circ\pi),\:\partial_r\rr\Big)\partial_r
\\
&=\Big(\<\nabla^{S_{r}}_\tau \tau,\:\nu\rr-\<\ode(\psi\circ\sigma),\:\nu\rr\Big)\nu
+\Big(-\dfrac{w'(r)}{w(r)}-\varphi'(r)\Big)\partial_r
\\
&=\vec{k}_{S_{r},\:\psi}+\vec{B}(\gamma),
\end{align*}
where $\overline{\nabla}$ is the covariant derivative of $(M^3_w,\:g_w)$.
\end{proof}

\begin{lemma}{\label{LeCamTi}} The application (\ref{CamTi}) is a diffeomorphism.
\end{lemma}
\begin{proof}
This application
\begin{itemize}
\item is smooth for being composed of smooth functions,
\item $\widetilde{t}(0)=0$, $\widetilde{t}(T)=\widetilde{T}$,
\item  its derivative is strictly positive for all $t$:
\begin{align*}
\widetilde{t}'(t)=\Big(\dfrac{w(r_0)}{w(R(t))}\Big)^2>0\text{\:for\:all\:}t\in[0,\:T).
\end{align*}
\end{itemize}
So by the inverse function theorem, it is a local diffeomorphism which, together the injectivity of the application, imply that it is a diffeomorphism.
\end{proof}

\begin{lemma}{\label{LeReCurGeoConDen}} Let $\gamma\in\mathcal{A}$ and let $\hat\gamma:=exp(\gamma,\:-(r-r_0)\partial_r)$ with $r=\pi(\gamma)$ and $r_0\in(0,\:\infty)$, then:
\begin{align*}
\vec{k}_{\widetilde\gamma,\:S_{r},\:\psi}
&=\dfrac{w^2(r_0)}{w^2(r)}\:\vec{k}_{\hat\gamma,\:S_{r_0},\:\psi}
\end{align*}
as vector field on $\mathbb{S}^2$.
\end{lemma}
\begin{proof} 
We note that $g_{S_r}=\dfrac{w^2(r)}{w^2(r_0)}g_{S_{r_0}}$, then:
\begin{align*}
\vec{k}_{\widetilde\gamma,\:S_{r},\:\psi}
&=\vec{k}_{\widetilde\gamma,\:S_{r}}-\nabla^{S_r}(\psi\circ\sigma)^\perp
=\dfrac{1}{\dfrac{w^2(r)}{w^2(r_0)}}\vec{k}_{\hat\gamma,\:S_{r_0}}-\nu(\psi\circ\sigma)\nu
\\
&=\dfrac{w^2(r_0)}{w^2(r)}\vec{k}_{\hat\gamma,\:S_{r_0}}-\vert \nu \vert^2_{S_{r_0}}\dfrac{\nu}{\vert \nu \vert_{S_{r_0}}}(\psi\circ\sigma)\dfrac{\nu}{\vert \nu \vert_{S_{r_0}}}
\\
&=\dfrac{w^2(r_0)}{w^2(r)}\vec{k}_{\hat\gamma,\:S_{r_0}}-\dfrac{w^2(r_0)}{w^2(r)}\dfrac{\nu}{\vert \nu \vert_{S_{r_0}}}(\psi\circ\sigma)\dfrac{\nu}{\vert \nu \vert_{S_{r_0}}}
\\
&=\dfrac{w^2(r_0)}{w^2(r)}\Big(\vec{k}_{\hat\gamma,\:S_{r_0}}-\dfrac{\nu}{\vert \nu \vert_{S_{r_0}}}(\psi\circ\sigma)\dfrac{\nu}{\vert \nu \vert_{S_{r_0}}}\Big)
\\
&=\dfrac{w^2(r_0)}{w^2(r)}\Big(\vec{k}_{\hat\gamma,\:S_{r_0}}-\nabla^{S_{r_0}}(\psi\circ\sigma)^\perp\Big)=\dfrac{w^2(r_0)}{w^2(r)}\vec{k}_{\hat\gamma,\:S_{r_0},\:\psi},
\end{align*}
 where $\nu$ is the normal field to the curve $\widetilde\gamma$ in the sphere $(S_r,\:g_r)$.
\end{proof}

Now, we are ready to give the proof of Theorem \ref{TeEyU}.

\begin{proof}{\bf (Theorem \ref{TeEyU})}
Given $\gamma$ by the expression (\ref{SoWar}), as the problem (\ref{PVIesfera}) has a unique solution for short times by Theorem \ref{TAnOaks}, (\ref{CamTi}) is a diffeomorphism by Lemma \ref{LeCamTi} and the system (\ref{SisRadio}) has a unique solution for being $B$ locally Lipschitz then, $\gamma$ is well defined on $\mathbb{S}^1\times[0,\:T)$ for any $T>0$ and there is a unique function with the definition (\ref{SoWar}). $\gamma$ is also smooth by being composed of smooth functions.

On the other hand:
\begin{align*}
\gamma(p,\:0)&=exp\Big(\widetilde\gamma(p,\:0),\:(R(0)-r_0)\partial_r\Big)=exp\Big(\widetilde\gamma(p,\:0),\:0\Big)
\\
&=\widetilde{\gamma}(p,\:0)=\widetilde{\gamma_0}(p)=\gamma_0(p),\:\forall\:p\in\mathbb{S}^1,
\end{align*}
and
\begin{align*}
\dfrac{\partial}{\partial t}\Big\vert_{t_1}\gamma(p,\:t)
&=\dfrac{\partial}{\partial t}\Big\vert_{t_1}exp\Big(\widetilde\gamma(p,\:\widetilde{t}),\:(R(t)-r_0)\partial_r\Big)
\\
&=\dfrac{\partial}{\partial t}\Big\vert_{t_1}exp\Big(\widetilde\gamma(p,\:\widetilde{t}_1),\:(R(t)-r_0)\partial_r\Big)
+\dfrac{\partial}{\partial t}\Big\vert_{t_1}exp\Big(\widetilde\gamma(p,\:\widetilde{t}),\:(R(t_1)-r_0)\partial_r\Big)
\\
&=\dfrac{\partial}{\partial t}\Big\vert_{t_1}\eta_{\widetilde\gamma(p,\:\widetilde{t}_1),\:\partial_r}\Big(R(t)-r_0\Big)
+exp_\star\Big(\dfrac{\partial}{\partial t}\Big\vert_{t_1}\widetilde\gamma(p,\:\widetilde{t}),\:(R(t_1)-r_0)\partial_r\Big)
\\
&=\eta_{\widetilde\gamma(p,\:\widetilde{t}_1),\:\partial_r}'\Big(R(t_1)-r_0\Big)R'(t_1)
+exp_\star\Big(\dfrac{\partial\widetilde t}{\partial t}\dfrac{\partial}{\partial\widetilde t}\Big\vert_{\widetilde t_1}\widetilde\gamma(p,\:\widetilde{t}),\:(R(t_1)-r_0)\partial_r\Big)
\\
&=R'(t_1)\partial_r\vert_{\gamma(p,\:t_1)}
+exp_\star\Big(\Big(\dfrac{w(r_0)}{w(R(t_1))}\Big)^2\big(\overrightarrow{k}_{S_{r_0},\:\psi}\big)_{\widetilde{\gamma}(p,\:\widetilde{t}_1)},\:(R(t_1)-r_0)\partial_r\Big)
\\
&=-B(R(t_1))\partial_r\vert_{\gamma(p,\:t_1)}
+\Big(\dfrac{w(r_0)}{w(R(t_1))}\Big)^2exp_\star\Big(
\big(\overrightarrow{k}_{S_{r_0},\:\psi}\big)_{\widetilde{\gamma}(p,\:\widetilde{t}_1)},\:(R(t_1)-r_0)\partial_r\Big)
\\
&=\vec{B}(\gamma(p,\:t_1))
+(\vec{k}_{S_{R(t_1),\:\psi}})_{\gamma(p,\:t_1)}
=(\vec{H}_\xi)_{\gamma(p,\:t_1)},
\end{align*}
where we have used Lemma \ref{LeReCurGeoConDen} and also Lemma \ref{LeExpCurMCurEs}, in the last equality, given that from (\ref{SoWar}) we have that $\gamma(\cdot,\:t)\in\mathcal{A}$ for all $t\in[0,\:T)$. We have denoted by $\eta_{\widetilde\gamma(p,\:\widetilde{t}_1),\:\partial_r}$ the geodesic that starts at $\widetilde\gamma(p,\:\widetilde{t}_1)$ and whose tangent vector at this point is $\partial_r$. Therefore (\ref{SoWar}) is a solution for problem (\ref{PVI}). 

To guarantee that (\ref{SoWar}) is the unique solution of the problem (\ref{PVI}) we may use the theorem of existence and unicity of solution for the MCF in higher codimension. The statement of this theorem was presented in a survey about this topic by K. Smoczyk \cite{Sm12}  as a special case of a theorem by R. Hamilton \cite{Ha82}. Let us remind the link between the $\xi$MCF and the MCF.
\end{proof}

\begin{corollary}{\label{CEaU}}
$Im\:\gamma(\cdot,\:t)\subset S_{R(t)},\:\forall\:t\in[0,\:T).$
\end{corollary}
\begin{proof}
By Theorem \ref{TeEyU} the solution is
\begin{align*}
\gamma(p,\:t)=exp\Big(\widetilde\gamma(p,\:\widetilde{t}(t)),\:(R(t)-r_0)\partial_r\Big),
\end{align*}
this expression implies that $Im\:\gamma(\cdot,\:t)\subset S_{R(t)},\:\forall\:t\in[0,\:T)$.
\end{proof}

\begin{corollary}\label{CEquiFEsRota}
Let $\gamma:\mathbb{S}^1\times[0,\:T)\rightarrow M^3_w$ be the maximal solution of the problem (\ref{PVI}) with initial condition $\gamma_0$, then
\begin{align*}
\widetilde\gamma(p,\:\widetilde t):=exp\Big(\gamma(p,\:t(\widetilde{t})),\:-(R(t(\widetilde t))-r_0)\partial_r\Big),
\end{align*}
is the unique solution defined in $\mathbb{S}^1 \times [0,\:\widetilde{T})$ for the problem (\ref{PVIesfera}) with initial condition $\widetilde{\gamma}_0$. The function $t:[0,\:\widetilde{T})\rightarrow [0,\:T)$ is the inverse function of (\ref{CamTi}).
\end{corollary}
\begin{proof}
From Theorem \ref{TeEyU} and Lemma \ref{LeCamTi}.
\end{proof}

\begin{remark}
It is highly relevant that $T$ could be the maximal time of the problem (\ref{PVI}) and $\widetilde{T}$ could not be the maximal time of the problem (\ref{PVIesfera}).
\end{remark}

\begin{theorem}{\bf (Preservation of the embedded property)}
Let $\gamma:\mathbb{S}^1\times[0,\:T)\rightarrow M^3_w$ be a solution of the problem (\ref{PVI}) such that $\gamma_0$ is an embedded curve then, $\gamma(\cdot,\:t)$ is an embedded curve for all $t\in[0,\:T)$.
\end{theorem}
\begin{proof}
From Theorem \ref{TeEyU} and Theorem \ref{TeCSPPreEm}.
\end{proof}

\begin{theorem}{\bf (Comparison Principle)} Let $\gamma_1:\mathbb{S}^1\times[0,\:T_1)\rightarrow M^3_w$,  $\gamma_2:\mathbb{S}^1\times[0,\:T_2)\rightarrow M^3_w$ be solutions of the initial value problem (\ref{PVI}) with initial conditions $\gamma_1(\cdot,\:0),\:\gamma_2(\cdot,\:0)\in\mathcal{A}$, respectively.  If $Im\:\gamma_1(\cdot,\:0)\cap Im\:\gamma_2(\cdot,\:0)=\emptyset$ then $Im\:\gamma_1(\cdot,\:t)\cap Im\:\gamma_2(\cdot,\:t)=\emptyset$ for all $t\in[0,\:\min\lbrace T_1,\:T_2\rbrace)$.
\end{theorem}
\begin{proof}
From Theorem \ref{TeEyU} and Theorem \ref{TeCSPComPrin}.
\end{proof}


\section{Finite maximal time}

In this section we analyse the different situations when the maximal time of the solution is finite, obtaining a proof for Theorem A.

\begin{theorem}{\label{TeTieFinCo}}
Let $\gamma_0\in\mathcal{A}$ and let $\gamma:\mathbb{S}^1\times[0,\:T)\rightarrow M^3_w$ be the unique maximal solution for the initial value problem (\ref{PVI}) with $\gamma_0$ as initial condition. If the maximal time $T$ is finite then, the solution collapses to a point.
\end{theorem}
\begin{proof}
We note that the continuous function $R(t)$ is either strictly decreasing, strictly increasing or constant. Thus, the $\lim_{t\rightarrow T}R(t)$ exists and if $T<\infty$ then, $\lim_{t\rightarrow T}R(t)<\infty$. Now, we have two possibilities:
\begin{itemize}
\item $\lim_{t\rightarrow T}R(t)=R(T)>0$: In this situation $\widetilde{T}<\infty$ and (\ref{SoWar}) only can have problems by part $\widetilde\gamma(\cdot,\:\widetilde t(t))$. Then, the maximal time for (\ref{PVIesfera}) is finite, moreover it is exactly $\widetilde{T}$, and by Theorem \ref{TAnOaks} the flow $\widetilde\gamma$ collapses to a point. Therefore, the flow $\gamma$ collapses to a point.
\item $\lim_{t\rightarrow T}R(t)=R(T)=0$: By Corollary \ref{CEaU} $Im\:\gamma(\cdot,\:t)\subset S_{R(t)},\:\forall\:t\in[0,\:T)$ then, the flow collapses to the pole $o$.
\end{itemize}
\end{proof}
\noindent Interestingly, we notice that in the second part of the proof, $\widetilde{T}$ may not be the maximal time of the flow $\widetilde{\gamma}$.

Now that we know that the solution collapses to a point, let's analyse which is the shape of the singularity.

\begin{theorem}\label{TeTieFinSin}
Let $\gamma_0\in\mathcal{A}$ and let $\gamma:\mathbb{S}^1\times[0,\:T)\rightarrow M^3_w$ be the unique maximal solution for the initial value problem (\ref{PVI}) with $\gamma_0$ as initial condition whose maximal time $T$ is finite. Then:
\begin{itemize}
\item[i)]  If the $\lim_{t\rightarrow T}R(t)\neq0$ the curve collapses to a spherical round point in the geodesic sphere $S_{R(T)}$.
\item[ii)] If the $\lim_{t\rightarrow T}R(t)=0$, that is, the curve collapses to the pole $o$ of the manifold $M^3_w$, and there exists $\widetilde{\varphi}\in C^1([0,\:\infty))$ such that $\widetilde\varphi\vert_{(0,\:\infty)}=\varphi$ then, a blow-up centered at the pole $o$ gives a limit flow by the $\psi$MCF in $(S_{r_0},\:g_{S_{r_0}},\:\psi)$ that $C^\infty$-subconverges, after a reparametrization of the curves, to a closed $\psi$-minimal curve.
\end{itemize}
\end{theorem}
\begin{proof}
By Theorem \ref{TeTieFinCo} the solution collapses to a point.
\begin{itemize}
\item[Case i)]  If $T<\infty$ and the $\lim_{t\rightarrow T}R(t)\neq0$ then, the maximal time of the solution of the problem (\ref{PVIesfera}), with initial condition $\widetilde\gamma_0$, is finite due to the relation between the flows. This, together with Theorem \ref{Zhu}, prove i). We note that, in this situation, the maximal time of (\ref{PVIesfera}) is $\widetilde{T}$.
\item[Case ii)] We need to study if $\widetilde{T}$ is finite or infinite. For this, we note that:
\begin{align}\label{PSinOri1}
\widetilde{T}=\int_0^T\Big(\dfrac{w(r_0)}{w(R(t))}\Big)^2\:dt\sim \int_0^T\dfrac{1}{R(t)^2}\:dt
\end{align}
we shall check this fact:
\begin{align*}
\lim_{t\rightarrow T} \dfrac{ \Big(\dfrac{w(r_0)}{w(R(t))}\Big)^2 }{ \dfrac{1}{R(t)^2} }
&=w(r_0)^2\lim_{t\rightarrow T} \dfrac{R(t)^2}{w(R(t))^2}
=w(r_0)^2\lim_{t\rightarrow T} \dfrac{2R(t)R'(t)}{2w(R(t))w'(R(t))R'(t)}
\\
&=w(r_0)^2\lim_{t\rightarrow T} \dfrac{R(t)}{w(R(t))w'(R(t))}
=w(r_0)^2\lim_{t\rightarrow T} \dfrac{R(t)}{w(R(t))}
\\
&=w(r_0)^2\lim_{t\rightarrow T} \dfrac{R'(t)}{w'(R(t))R'(t)}
=w(r_0)^2\lim_{t\rightarrow T} \dfrac{1}{w'(R(t))}
\\
&=w(r_0)^2\in(0,\:\infty),
\end{align*}
where we have used $w(0)=0$, $w'(0)=1$ and the L'Hôpital's rule twice.

Now, we are going to study the second integral of (\ref{PSinOri1}):
\begin{align*}
\int_0^T&\dfrac{1}{R(t)^2}\:dt=\int_0^T\dfrac{R'(t)}{R'(t)R(t)^2}\:dt
=\int_0^T\dfrac{R'(t)}{\Big(-\dfrac{w'(R(t))}{w(R(t))}-\varphi'(R(t)) \Big)R(t)^2}\:dt
\\
&=\int_{R(0)}^{R(T)}\dfrac{1}{\Big(-\dfrac{w'(x)}{w(x)}-\varphi'(x) \Big)x^2}\:dx
=\int_{r_0}^{0}\dfrac{1}{\Big(-\dfrac{w'(x)}{w(x)}-\varphi'(x) \Big)x^2}\:dx
\\&=\int_{0}^{r_0}\dfrac{1}{\Big(\dfrac{w'(x)}{w(x)}+\varphi'(x) \Big)x^2}\:dx,
\end{align*}
we remark that $R'(t)<0$ for all $t\in[0,\:T)$, as otherwise the hypothesis $\lim_{t\rightarrow T}R(t)=0$ is impossible.
Going back to the integral:
\begin{align*}
\lim_{x\rightarrow 0^+}&\dfrac{\dfrac{1}{\Big(\dfrac{w'(x)}{w(x)}+\varphi'(x) \Big)x^2}}{\dfrac{1}{x}}
=\lim_{x\rightarrow 0^+}\dfrac{1}{\Big(\dfrac{w'(x)}{w(x)}+\varphi'(x) \Big)x}
=\lim_{x\rightarrow 0^+}\dfrac{1}{\dfrac{x}{w(x)}}
\\
&=\lim_{x\rightarrow 0^+}\dfrac{w(x)}{x}
=\lim_{x\rightarrow 0^+}w'(x)=1\in(0,\:\infty),
\end{align*}
we have used the L'Hôpital's rule and also that $\varphi$ has a $C^1$-extension to $[0,\:\infty)$, so
\begin{align}\label{PSinOri2}
\int_{0}^{r_0}\dfrac{1}{\Big(\dfrac{w'(x)}{w(x)}+\varphi'(x) \Big)x^2}\:dx\sim\int_{0}^{r_0}\dfrac{1}{x}\:dx.
\end{align}
Therefore, from (\ref{PSinOri1}) and (\ref{PSinOri2}), $\widetilde{T}=\infty$ and if we perform a rescaling on the singularity we obtain exactly the flow $\widetilde{\gamma}$ as rescaling flow of $\gamma$, so the rescaling flow subconverges, in the sense of Theorem \ref{TViMi}, to a closed $\psi$-minimal curve in $(S_{r_0},\:g_{S_{r_0}},\:\psi)$ by this theorem. Note that the rescaling flow is given by
\begin{align*}
\exp_{\gamma(\cdot,\:t)}\big((r_0-\:R(t))\partial_r\big),
\end{align*}
and this flow is the flow $\widetilde{\gamma}$ by Corollary \ref{CEquiFEsRota}.
\end{itemize}
\end{proof}

\begin{remark}
The case ii) of the previous theorem is not true if the density $\varphi$ does not have a $C^0$-extension to $[0,\:\infty)$. For example, if we consider $\varphi(r):=-\ln w(r)-\dfrac{1}{r}$ then:
\begin{align*}
\widetilde{T}\sim \int_{0}^{r_0}\dfrac{1}{\Big(\dfrac{w'(x)}{w(x)}+\varphi'(x) \Big)x^2}\:dx
=\int_{0}^{r_0}\dfrac{1}{\dfrac{1}{x^2}x^2}\:dx=r_0<\infty
\end{align*}
Therefore, if the solution collapses to the origin, the rescaling flow of $\gamma$ faces two possibilities: either converges to the curve $\widetilde\gamma(\cdot,\widetilde{T})$, if $\widetilde{T}$ is not the maximal time of $\widetilde\gamma$, or, if $\widetilde T$ is the maximal time of $\widetilde\gamma$, we need to make a new rescaling in the sphere and with this second rescaling we obtain that the curve converges to a round point by Theorem \ref{Zhu}.
\end{remark}

\section{Infinite maximal time }

In this section we study the situation in which the maximal time  of the solution for the problem (\ref{PVI}) is infinite, obtaining a proof for Theorem B and another for Theorem C. In this situation, there are two delicate scenarios: when the solution goes to infinity and when the solution collapses to the pole. This fact motivates splitting the study in three cases: when the solution collapses to the pole, when the solution is in a bounded region $0<C_1\leq R(t)\leq C_2$ and when the solution goes to infinity.

First, we are going to tackle the following question: Could the solution collapse to the pole in infinite time?

\begin{proposition}\label{PropoColapOrigenConExtenTF}
Let $\gamma_0\in\mathcal{A}$ and let  $\gamma:\mathbb{S}^1\times[0,\:T)\rightarrow M^3_w$ be the unique maximal solution of the initial value problem (\ref{PVI}) with $\gamma_0$ as initial condition. If this solution collapses to the pole in the maximal time and there exists $\widetilde{\varphi}\in C^1([0,\:\infty))$ such that $\widetilde\varphi\vert_{(0,\:\infty)}=\varphi$ then, the maximal time of the flow is finite.
\end{proposition}
\begin{proof}
As $\lim_{t\rightarrow T}R(t)=0$ then $\lim_{t\rightarrow T}w(R(t))=0$, $\lim_{t\rightarrow T}w'(R(t))=1$ and $\lim_{t\rightarrow T}\varphi'(R(t))=\varphi'(0)$, the last equality relies on the hypothesis about $\varphi$. So there are $t^\star\in[0,\:T)$ and $C>0$ such that 
\begin{align*}
R'(t)\leq -C,\:\forall\:t\in[t^\star,\:T).
\end{align*}
If we integrate this inequality, we obtain:
\begin{align*}
R(t_2)-R(t_1)\leq -C(t_2-t_1),\:\forall\:t_2\geq t_1\geq t^\star,
\end{align*}
so
\begin{align*}
C(t_2-t_1)\leq R(t_2)+C(t_2-t_1)\leq R(t_1)<\infty,\:\forall\:t_2\geq t_1\geq t^\star,
\end{align*}
if we take $t_1=t^\star$ and $t_2\rightarrow T$:
\begin{align*}
C(T-t^\star)\leq R(t^\star)<\infty,
\end{align*}
therefore $T<\infty$.
\end{proof}

Having this in mind, the answer to the question is: No if the density $\varphi$ has $C^1$-extension to $[0,\:\infty)$. However, if the density $\varphi$ does not have a $C^1$-extension to $[0,\:\infty)$, the solution could collapse to the pole in infinite time. For example, let $(\mathbb{R}^3,\:g_{\mathbb{R}^3},\:\xi)$ be the ambient manifold with $\varphi(r)=-\ln(r)+\dfrac{1}{2}r^2$. In this situation, the derivative of $R(t)$ satisfies:
\begin{align*}
R'(t)&=-\dfrac{1}{R(t)}-\varphi'(R(t))=-\dfrac{1}{R(t)}+\dfrac{1}{R(t)}-R(t)=-R(t)
\\
\Rightarrow R(t)&=r_0 e^{-t}.
\end{align*}
Then, if the flow collapses to the pole, the maximal time of the solution is infinite.

Now, we continue studying the other cases. For the case where the solution is in a bounded region, we need some previous lemmas. Let $s$ be the arc length parameter of $\gamma(\cdot,\:t)$ and let $\ws$ be the arc length parameter of $\wgamma$.
  
\begin{lemma}{\label{LeDeCur}}
\begin{itemize}
\item[]
\item $\partial_s=\dfrac{w(r_0)}{w(R(t))}\partial_{\widetilde{s}}$,
\item $\partial_s^n\Big(k_{\gamma(\cdot,\:t),\:S_{R(t)},\:\psi}\Big)=\Big(\dfrac{w(r_0)}{w(R(t))}\Big)^{n+1}\partial_{\widetilde{s}}^n\Big(k_{\widetilde{\gamma}(\cdot,\:\widetilde{t}),\:S_{r_0},\:\psi}\Big)$.
\end{itemize}
\end{lemma}
\begin{proof}
The relation between $ds$ and $d\widetilde{s}$ is
\begin{align*}
ds&=\vert\partial_\alpha\gamma(\alpha,\:t) \vert_{g_{S_{R(t)}}} d\alpha=\dfrac{w(R(t))}{w(r_0)}\vert\partial_\alpha\gamma(\alpha,\:t) \vert_{g_{S_{r_0}}} d\alpha=\dfrac{w(R(t))}{w(r_0)}d\ws.
\end{align*}
Therefore
\begin{align*}
\partial_s=\dfrac{w(r_0)}{w(R(t))}\partial_{\widetilde{s}}.
\end{align*}
On the other hand, the relation between the mean curvature vectors with density is given from Lemma \ref{LeReCurGeoConDen}  by
\begin{align*}
\vec{k}_{\gamma(\cdot,\:t),\:S_{R(t)},\:\psi}
&=\dfrac{w^2(r_0)}{w^2(R(t))}\vec{k}_{\wgamma,\:S_{r_0},\:\psi},
\end{align*}
then
\begin{align*}
k_{\gamma(\cdot,\:t),\:S_{R(t)},\:\psi}\:\nu
&=\dfrac{w^2(r_0)}{w^2(R(t))} k_{\wgamma,\:S_{r_0},\:\psi}\dfrac{\nu}{\vert \nu\vert_{S_{r_0}}},
\\
k_{\gamma(\cdot,\:t),\:S_{R(t)},\:\psi}\:\nu
&=\dfrac{w^2(r_0)}{w^2(R(t))} k_{\wgamma,\:S_{r_0},\:\psi}\dfrac{\nu}{\dfrac{w(r_0)}{w(R(t))}},
\\
k_{\gamma(\cdot,\:t),\:S_{R(t)},\:\psi}\:\nu
&=\dfrac{w(r_0)}{w(R(t))} k_{\wgamma,\:S_{r_0},\:\psi}\:\nu,
\\
k_{\gamma(\cdot,\:t),\:S_{R(t)},\:\psi}
&=\dfrac{w(r_0)}{w(R(t))} k_{\wgamma,\:S_{r_0},\:\psi},
\end{align*}
and we obtain that
\begin{align*}
\partial_s^n\Big(k_{\gamma(\cdot,\:t),\:S_{R(t)},\:\psi}\Big)=\Big(\dfrac{w(r_0)}{w(R(t))}\Big)^{n+1}\partial_{\widetilde{s}}^n\Big(k_{\widetilde{\gamma}(\cdot,\:\widetilde{t}),\:S_{r_0},\:\psi}\Big).
\end{align*}
\end{proof}

Also, we need an expression that links $\partial_s^n \big(k_{\gamma(\cdot,\:t),\:S_{R(t)}}\big)$ with $\partial_s^n\big( k_{\gamma(\cdot,\:t),\:S_{R(t)},\:\psi}\big)$. We may borrow the expression given in (51) of \cite{MiVi16}:
\begin{align}\label{ReVarCurOrDen}
\partial_s^n \big(k_{\gamma(\cdot,\:t),\:S_{R(t)}}\big)
&=\partial_s^n\big( k_{\gamma(\cdot,\:t),\:S_{R(t)},\:\psi}\big)+\nabla^{n+1}\psi(\partial_s,\:\cdots,\:\partial_s,\:\nu)
\nonumber\\ &\quad+\sum_{n,\:1}^{1,\:n-1}c_{i,\:J,\:K}\: k_{\gamma(\cdot,\:t),\:S_{R(t)},\:\psi}^i\:\partial_s^J\big( k_{\gamma(\cdot,\:t),\:S_{R(t)},\:\psi}\big)\:C(\nabla^K \psi),
\end{align}
we denote by $\nabla$ the covariant derivative in $S_{R(t)}$, which is independent of $t$, $J=(j_1,\:\cdots,\:j_q)$,\:$j_1\leq\:j_2\:\leq\cdots\leq j_q$ is an ordered multi-index and we denote by $\vert J\vert:=j_1+\cdots+j_q$, $d(J):=q$,\:$\partial_s^Jx:=\partial_s^{j_1}\:x\cdots\partial_s^{j_q}\:x$, $\nabla^Jx:=\nabla^{j_1}x\otimes\cdots\otimes\nabla^{j_q}x$. About the summation notation we only need to know that given $\sum_{m,\:r}^{s,\:t}$ then $i+\vert J\vert+d(J)+\vert K\vert=m+r$. And by $C(\nabla^K \psi)$ we denote $\nabla^K \psi$ acting on $\vert K\vert$ copies of $\partial_s$ or/and $\nu$.

We shall denote by $\overline\nabla$ the covariant derivative of the Riemannian manifold $(M^3_w,\:g_w)$.

\begin{lemma} Given $\gamma\in\mathcal{A}$
\begin{align}
\left\lbrace
 \begin{array}{ccccc}
\overline{\nabla}_\tau \tau& = &  & k\nu & -\dfrac{w'}{w}\partial_r     \\
\overline{\nabla}_\tau \nu & = & -k\tau &  &       \\
\overline{\nabla}_\tau \partial_r & = & \dfrac{w'}{w}\tau  & &\\
\end{array} \right.
\end{align}
where $\lbrace\tau,\:\nu,\:\partial_r\rbrace$ is an orthonormal frame over the curve with $\tau$, the unit tangent vector to $\gamma$, $\nu$, unit normal to $\gamma$ and tangent to the geodesic sphere $S_r$ where the curve is contained, and with $k$, the geodesic curvature of the curve $\gamma$ as curve of the sphere $(S_r,\:g_r)$.
\end{lemma}
\begin{proof} See \cite{On83}:
\begin{align*}
\overline{\nabla}_\tau \tau&=\< \overline{\nabla}_\tau \tau,\:\nu\right>\nu+\< \overline{\nabla}_\tau \tau,\:\partial_r\right>\partial_r
=\< \nabla^{S_r}_\tau \tau,\:\nu\right>\nu+\< -\dfrac{\left\langle \tau,\:\tau\right\rangle}{w}\overline{\nabla} w,\:\partial_r\right>\partial_r
\\
&=k\nu-\dfrac{w'}{w}\partial_r,
\\
\overline{\nabla}_\tau \nu&=\<\overline{\nabla}_\tau \nu,\:\tau\right>\tau+\< \overline{\nabla}_\tau \nu,\:\partial_r\right>\partial_r
=-\< \nu,\:\overline{\nabla}_\tau\tau\right>\tau+\<  -\dfrac{\left\langle \tau,\:\nu\right\rangle}{w}\overline\nabla w,\:\partial_r\right>\partial_r
\\
&=-k\tau,
\\
\overline{\nabla}_\tau \partial_r&=\dfrac{\partial_r(w)}{w}\tau=\dfrac{w'}{w}\tau.
\end{align*}

\end{proof}

\begin{theorem}\label{TInComAsin}
Let $\gamma_0\in\mathcal{A}$ and let  $\gamma:\mathbb{S}^1\times[0,\:T)\rightarrow M^3_w$ be the unique maximal solution of the initial value problem (\ref{PVI}) with $\gamma_0$ as initial condition, if the solution exists for all $t$, that is, $T=\infty$, then:
\begin{itemize}
\item[i)] If $0<C_1\leq R(t)\leq C_2$ with $C_1,\:C_2$ some constants for all $t\in[0,\:\infty)$, the flow $C^\infty$-subconverges, after a reparametrization of the curves $\gamma(\cdot,\:t)$, to a closed $\psi$-minimal spherical curve contained in the $B$-minimal geodesic sphere $S_{\lim_{t\rightarrow\infty} R(t)}$.
\item[ii)] If $\lim_{t\rightarrow T}R(t)=\infty$ then, $B(r)<0$ for all $r\in[r_0,\:\infty)$ is a necessary condition.
\begin{itemize}
\item[ii.a)] If $M^3_w$ is parabolic and the $\liminf_{r\rightarrow\infty}B(r)$ is finite then, the flow topologically subconverges to $\gamma_\infty:\mathbb{S}^1\rightarrow [0,\:\infty]\times\mathbb{S}^2,\:p\mapsto (\infty,\:\chi(p))$ where $\chi:\mathbb{S}^1\rightarrow S_{r_0}$ is a smooth embedded closed $\psi$-minimal curve in $(S_{r_0},\:g_{S_{r_0}},\:\psi)$.
\item[ii.b)]  If $M^3_w$ is hyperbolic and the $\limsup_{r\rightarrow\infty}B(r)\neq 0$ then, the flow either
\begin{itemize}
\item topologically converges to $\gamma_\infty:\mathbb{S}^1\rightarrow [0,\:\infty]\times\mathbb{S}^2,\:p\mapsto (\infty,\:\widetilde{\gamma}(p,\:\widetilde{T}))$, which is a curve contained in $\mathbb{S}_\infty\equiv \lbrace\infty\rbrace\times\mathbb{S}^2\subset [0,\:\infty]\times\mathbb{S}^2$ the infinite radius sphere,
\item or topologically converges to a point $p_\infty$ in $\mathbb{S}_\infty\equiv \lbrace\infty\rbrace\times\mathbb{S}^2\subset [0,\:\infty]\times\mathbb{S}^2$ the infinite radius sphere.
\end{itemize}
\end{itemize}
\end{itemize}
\end{theorem}
\begin{proof}
\begin{itemize}
\item[]
\item[Case i)]
If $0<C_1\leq R(t)\leq C_2$ for all $t\in[0,\:\infty)$ and the maximal time of the solution of (\ref{PVI}) is infinite then,  $\widetilde{T}=\infty$:
\\
 The $\lim_{t\rightarrow \infty}R(t)=R_\infty\in\mathbb{R}$ $\Rightarrow$ $\lim_{t\rightarrow \infty}w(R(t))=w(R_\infty)\in\mathbb{R}$, so $\forall\:\epsilon>0\:\exists\: t_\epsilon>0$ such that 
$\vert w(R(t))-w(R_\infty) \vert<\epsilon$ for all $t\geq t_\epsilon$. Therefore
\begin{align*}
\widetilde{T}&=\int_0^\infty\Big(\dfrac{w(r_0)}{w(R(t))}\Big)^2\:dt
\\
&\geq\int_0^{t_\epsilon}\Big(\dfrac{w(r_0)}{w(R(t))}\Big)^2\:dt +\int_{t_\epsilon}^\infty\Big(\dfrac{w(r_0)}{w(R_\infty)+\epsilon}\Big)^2\:dt=\infty\Rightarrow \:\widetilde{T}=\infty.
\end{align*}
As $\widetilde{t}:[0,\infty)\longrightarrow [0,\infty)$ is a diffeomorphism by Lemma \ref{LeCamTi}, we obtain that the behaviour of the flows (\ref{PVI}) and (\ref{PVIesfera}) is the same.

From Step 4 on page 23 of \cite{MiVi16}, Lemma \ref{LeDeCur} and the hypothesis about $R(t)$, we obtain that
\begin{align}\label{CurNula}
\partial_s^n\Big(k_{\gamma(\cdot,\:t),\:S_{R(t)},\:\psi}\Big)& \text{\:converges\:uniformly\:to\:zero\:when\:}t\rightarrow\infty
\nonumber\\ &\text{\:for\:every}\:n\in\mathbb{N}.
\end{align}  

On the other hand,
\begin{align*}
\partial_s \gamma&=\tau,
\\
\partial_s^2 \gamma&=\overline{\nabla}_\tau \tau=k\nu-\dfrac{w'}{w}\partial_r,
\\
\partial_s^3 \gamma&=\overline{\nabla}_\tau \Big(k\nu-\dfrac{w'}{w}\partial_r\Big)
=\partial_s k\:\nu+k\overline{\nabla}_\tau\nu-\partial_s\Big(\dfrac{w'}{w}\Big)\partial_r-\dfrac{w'}{w}\overline{\nabla}_\tau\partial_r
\\&=\partial_s k\:\nu-k^2\tau-\big(\dfrac{w'}{w}\big)^2\tau
=\Big(-k^2-\big(\dfrac{w'}{w}\big)^2\Big)\tau+\partial_s k\:\nu,
\\
\partial_s^4 \gamma&=\overline{\nabla}_\tau\Big[\Big(-k^2-\big(\dfrac{w'}{w}\big)^2\Big)\tau+\partial_s k\:\nu\Big]
\\
&=-2k\partial_s k\:\tau+\Big(-k^2-\big(\dfrac{w'}{w}\big)^2\Big)\overline{\nabla}_\tau\tau+\partial_s^2 k\:\nu+\partial_s k\:\overline{\nabla}_\tau\nu
\\
&=-2k\partial_s k\:\tau+\Big(-k^2-\big(\dfrac{w'}{w}\big)^2\Big)\Big(k\nu-\dfrac{w'}{w}\partial_r\Big)+\partial_s^2 k\:\nu
-k\partial_s k\:\tau
\\
&=-3k\partial_s k\:\tau+\Big(-k^3-k\big(\dfrac{w'}{w}\big)^2+\partial_s^2 k\Big)\nu
+\Big(k^2\dfrac{w'}{w}+\big(\dfrac{w'}{w}\big)^3\Big)\partial_r,
\end{align*}
we can obtain an expression for $\partial_s^n\gamma$ of the form:
\begin{align}
\partial_s^n\gamma&=f_n (\dfrac{w'}{w},\:k,\:\partial_s k,\cdots,\:\partial_s^{n-3}k)\tau+g_n (\dfrac{w'}{w},\:k,\:\partial_s k,\cdots,\:\partial_s^{n-2}k)\nu
\nonumber\\&\quad+
h_n(\dfrac{w'}{w},\:k,\:\partial_s k,\cdots,\:\partial_s^{n-4}k)\partial_r,\:n\geq 1,
\end{align}
where $f_n,\:g_n$ and $h_n$ are polynomials in $\dfrac{w'}{w},\:k,\:\partial_s k,\cdots,\:\partial_s^{j}k$ with $j\leq n-2$ where all monomials have degree $n-1$, which is obtained counting $\partial_s^i k$ as $i+1$. With this perspective the application $\gamma$ has the form $\gamma=\gamma(s_t,\:t)$ and now we consider the change of parameter:
\begin{align*}
\left.
 \begin{array}{ccc}
[0,\:\text{Length}_t]& \longrightarrow &   [0,\:1]  \\
s_t & \longmapsto&\alpha:= \dfrac{s_t}{L_t}
\end{array} \right.
\end{align*}
and we denote by $\hat{\gamma}$ the reparametrization of the curves with the parameter $\alpha$, that is $\hat{\gamma}(\alpha,\:t)=\gamma(s_t(\alpha),\:t)$. We notice that
\begin{align*}
\dfrac{\partial s_t }{\partial\alpha}=L_t\quad\text{ and }\quad
\dfrac{\partial^n s_t }{\partial\alpha}=0,\:\forall\: n\geq 2.
\end{align*}
Then we have that:
\begin{align*}
\partial^n_\alpha\:\hat{\gamma}&=L_t^n\: \partial^n_{s_t}\gamma,\:\forall\: n=1,\:2,\:\cdots
\end{align*}
and therefore:
\begin{align}\label{ConDerNdeLaCurva}
\vert \partial^n_\alpha\:\hat{\gamma}\vert =\vert L_t^n\: \partial^n_{s_t}\gamma\vert=L_t^n \vert  \partial^n_{s_t}\gamma\vert = L_t^n  \sqrt{f_n^2+g_n^2+h_n^2},\:\forall\: n=1,\:2,\:\cdots
\end{align}

We may also obtain the following bound for the length of the curves:
\begin{align*}
e^{\xi(\gamma_t)}=e^{\varphi\circ\pi(\gamma_t)+\psi\circ\sigma(\gamma_t)}=e^{\varphi\circ\pi(\gamma_t)}e^{\psi\circ\sigma(\gamma_t)}\geq \:\min_{r\in[C_1,\:C_2]}e^{\varphi(r)}\:\min_{p\in\mathbb{S}^2}e^{\psi(p)},
\end{align*}
we denote by D the last expression, then:
\begin{align}\label{ProCotaLong}
L_t&=\int_{\mathbb{S}^1}ds_t=\int_{\mathbb{S}^1}\dfrac{e^{\xi(\gamma_t)}}{e^{\xi(\gamma_t)}}ds_t\leq \dfrac{1}{D}\int_{\mathbb{S}^1}e^{\xi(\gamma_t)}ds_t
\nonumber\\&=\dfrac{1}{D}L_\xi(\gamma_t)\leq \dfrac{1}{D}L_\xi(\gamma_0),
\end{align}
as $L_\xi(\gamma_t)$ decreases throughout the flow, so $L_t$ is bounded independly of $t$.

On the other hand, the hypothesis $C_1\leq R(t)\leq C_2$ for all $t\in[0,\:\infty)$ implies that  $\dfrac{w'}{w}$ and $w$ are bounded. This fact, together with (\ref{CurNula}) and (\ref{ReVarCurOrDen}), imply that
\begin{align}\label{AcotaNormaFuncionesDerivada}
\sqrt{f_n^2+g_n^2+h_n^2}\text{\:is\:bounded\:independly\:of\:t.}
\end{align}
Therefore, from (\ref{ConDerNdeLaCurva}), (\ref{ProCotaLong}) and (\ref{AcotaNormaFuncionesDerivada}):
\begin{align}
\vert \partial^n_\alpha\:\hat{\gamma}(\alpha,\:t)\vert&\leq C_n,\:\forall\:(\alpha,\:t)\in[0,\:1]\times[0,\:\infty),\:\forall\: n=1,\:2,\:\cdots
\end{align}
with $C_n$ independent of $(\alpha,\:t)$. The case $n=0$ is immediate from the hypothesis about $R(t)$. We might use the Arzela-Ascoli theorem to conclude that there is a family $\lbrace \hat{\gamma}(\cdot,\:t_m)\rbrace_{m\in\mathbb{N}}$, $t_m\rightarrow\infty$, such that $C^\infty$-converges to a limit closed regular curve $\hat{\gamma}_\infty$. To obtain this result, we use a diagonal type argument.

The limit curve is regular because of Lemma 8 of \cite{MiVi16}. This implies that, in our situation, $L_t\geq c$ for all $t\in[0,\:\infty)$ for some constant $c$, so $\vert\partial_\alpha\hat{\gamma}\vert=L_t\geq c$ and the limit curve $\hat\gamma_\infty$ is regular.

We note that the geodesic sphere whose radius is $R_\infty=\lim_{t\rightarrow\infty}R(t)$ is B-minimal and from (\ref{CurNula}) $k_{\gamma(\cdot,\:t),\:S_{R(t)},\:\psi}$ converges uniformly to zero when $t\rightarrow\infty$, then $\hat{\gamma}_\infty$ is a $\psi$-minimal curve contained in the B-minimal geodesic sphere $S_{R_\infty}$.
\item[Case ii)] As the $\lim_{t\rightarrow T}R(t)=\infty$ then $R'(t)> 0$ for all $t\in[0,\infty)$. Otherwise, there is $t^\star\in [0,\:\infty)$ such that $S_{R(t^\star)}$ is a $B$-minimal sphere and the solution $\gamma(\cdot,\:t)$ is contained in $\mathbb{B}_{R(t^\star)}$ for all time, therefore $\lim_{t\rightarrow T}R(t)\leq R(t^\star)$. As a consequence, we obtain that $B(R)<0$ for all $R\in[r_0,\infty)$, due to $R'(t)=-B(R(t))$.

The hypothesis of the case ii.a) about the function $B$ implies that exists a constant $C>0$ such that $-C\leq B(r)<0$ for all $r\in[r_0,\infty)$. Besides, the hypothesis of the case ii.b) about the function $B$ implies that exists a constant $C>0$ such that $B(r)\leq-C<0$ for all $r\in[r_0,\infty)$. As a summary:
\begin{align*}
\left. \begin{array}{ccc}
\text{In\:the\:case\:ii.a)\:}& -\dfrac{1}{B(r)}\geq \dfrac{1}{C}, & \text{for\:all\:}r\in[r_0,\:\infty). \\
\text{In\:the\:case\:ii.b)\:}& \dfrac{1}{C}\geq -\dfrac{1}{B(r)},&\text{for\:all\:}r\in[r_0,\:\infty).\end{array} \right.
\end{align*}
On the other hand, we could obtain the following expression for $\widetilde{T}$:
\begin{align}
\widetilde{T}&=\int_0^\infty\Big(\dfrac{w(r_0)}{w(R(t))}\Big)^2\:dt=\int_0^\infty\dfrac{Area(S_{r_0})}{Area(S_{R(t)})}\:dt
\nonumber\\
&=Area(S_{r_0})\int_0^\infty\dfrac{R'(t)}{R'(t)\:Area(S_{R(t)})}\:dt
\nonumber\\
&=-Area(S_{r_0})\int_0^\infty\dfrac{R'(t)}{B(R(t))\:Area(S_{R(t)})}\:dt
\nonumber\\
&=-Area(S_{r_0})\int_{r_0}^\infty\dfrac{1}{B(r)\:Area(S_{r})}\:dr,
\end{align}
we realise that $R:[0,\:\infty)\rightarrow [r_0,\:\infty)$ defines a diffeomorphism on its image.

\begin{itemize}
\item[Case ii.a)] $M^3_w$ is parabolic and $-\dfrac{1}{B(r)}\geq \dfrac{1}{C}\:\text{for\:all\:}r\in[r_0,\:\infty)$: $\widetilde{T}=\infty$.
\begin{align*}
\widetilde{T}&=-Area(S_{r_0})\int_{r_0}^\infty\dfrac{1}{B(r)\:Area(S_{r})}\:dr
\\&\geq 
\dfrac{Area(S_{r_0})}{C}\int_{r_0}^\infty\dfrac{1}{Area(S_{r})}\:dr=\infty,
\end{align*}
the last equality is true because the Riemannian manifold $M^3_w$ is parabolic (see \cite{Gri99}).
\\
As we have that $\widetilde{t}:[0,\infty)\longrightarrow [0,\infty)$ is a diffeomorphism, we conclude that the behaviour of both flows is the same in infinite time. And the behaviour of the flow $\widetilde\gamma$, with infinite maximal time, is given by Theorem \ref{TViMi}.
\item[Case ii.b)] $M^3_w$ is hyperbolic and $\dfrac{1}{C}\geq -\dfrac{1}{B(r)}\:\text{for\:all\:}r\in[r_0,\:\infty)$: $\widetilde{T}<\infty$.
\begin{align*}
\widetilde{T}&=-Area(S_{r_0})\int_{r_0}^\infty\dfrac{1}{B(r)\:Area(S_{r})}\:dr
\\
&\leq 
\dfrac{Area(S_{r_0})}{C}\int_{r_0}^\infty\dfrac{1}{Area(S_{r})}\:dr<\infty,
\end{align*}
we note that the last equality is true because the Riemannian manifold $M^3_w$ is hyperbolic (see \cite{Gri99}).
\\
As we have that $\widetilde{t}:[0,\infty)\longrightarrow [0,\widetilde{T})$ is a diffeomorphism with $\widetilde{T}<\infty$, we get that the behaviour of the flow $\gamma$ in infinite time is the behaviour of the flow $\widetilde{\gamma}$ in finite time. This flow in $\widetilde{t}=\widetilde{T}$ has two options: either the flow is defined, $\widetilde{\gamma}(\cdot,\widetilde{T})$ is a smooth curve, or the flow collapses to a point by Theorem \ref{TAnOaks}. We remark that in the first situation $\widetilde{T}$ is not the maximal time of the flow (\ref{PVIesfera}).
\end{itemize}
\end{itemize}
\end{proof}

We note that Proposition \ref{PropoColapOrigenConExtenTF} together with Theorem \ref{TInComAsin} i) give us a proof for Theorem B and Theorem \ref{TInComAsin} ii) is essentially Theorem C.

\begin{remark}
The case ii.a) is not true if we eliminate the condition on the function $B$. We may find situations where $M^3_w$ is parabolic, $\liminf_{r\rightarrow\infty} B(r)=-\infty$ and $\widetilde{T}<\infty$. For example, let $(M^3_w,\:g_{w},\:\xi)$ be a smooth rotationally symmetric space such that $w\vert_{[C,\:\infty)}(r)=\sqrt{r}$ and $\varphi(r)=-\dfrac{1}{2}\ln(r)-\dfrac{r^2}{2}$. Then, $M^3_w$ is parabolic,
\begin{align*}
\int_0^\infty \dfrac{dr}{Area(S_r)}&=\int_0^\infty \dfrac{dr}{4\pi w(r)^2}=\int_0^C \dfrac{dr}{4\pi w(r)^2}
+\int_C^\infty \dfrac{dr}{4\pi r}=\infty,
\end{align*}
and
\begin{align*}
B(r)=\dfrac{w'}{w}+\varphi'(r)=\dfrac{1}{2r}-\dfrac{1}{2r}-r=-r<0,\:\text{for\:all\:}r\in(C,\:\infty),
\end{align*}
so
\begin{align*}
\liminf_{r\rightarrow\infty} B(r)=-\infty.
\end{align*}
Given $\gamma_0\in\mathcal{A}$, such that $\pi(Im\:\gamma_0)>C$, as initial condition then, the system for $R$ is as follows:
\begin{align*}
\left\lbrace \begin{array}{ccc}
R'(t) & = & R(t), \\
R(0) & = & r_0, \end{array} \right.
\end{align*}
thus $R(t)=r_0e^t$. 

\noindent On the other hand, if we assume that $\psi\equiv0$ and that the area enclosed by the curve $\gamma_0$ is $Area(S_{r_0})/2$ then, $T=\infty$. 

\noindent Bearing all this in mind, the value of $\widetilde{T}$ is
\begin{align*}
\widetilde{T}&=\int_0^\infty\Big(\dfrac{w(r_0)}{w(R(t))}\Big)^2\:dt=
\int_0^\infty\dfrac{r_0}{r_0e^t}\:dt<\infty.
\end{align*}
Therefore, the flow
\begin{itemize}
\item topologically converges to $\gamma_\infty:\mathbb{S}^1\rightarrow [0,\:\infty]\times\mathbb{S}^2,\:p\mapsto (\infty,\:\widetilde{\gamma}(p,\:\widetilde{T}))$,
\item or topologically converges to a point $p_\infty\in\mathbb{S}_\infty\equiv \lbrace\infty\rbrace\times\mathbb{S}^2\subset [0,\:\infty]\times\mathbb{S}^2$.
\end{itemize}
\end{remark}

\begin{remark} In the case ii.b) the situation is analogous: it is not true if we eliminate the condition on the function $B$. We may find situations where $M^3_w$ is hyperbolic, $\limsup_{r\rightarrow\infty}B(r)= 0$ and $\widetilde{T}=\infty$. For example, let $(\mathbb{R}^3,\:g_{\mathbb{R}^3},\:\xi)$ be the 3-dimensional Euclidean space with a density $\xi$ such that $\varphi(r)=-2\ln(r)$. Then, $\mathbb{R}^3$ is a hyperbolic manifold and
\begin{align*}
B(r)= \dfrac{1}{r}+\varphi'(r)=-\dfrac{1}{r}<0,\:\text{for\:all\:}r\in(0,\:\infty),
\end{align*}
so
\begin{align*}
\limsup_{r\rightarrow\infty}\: B(r)=0.
\end{align*}
Given $\gamma_0\in\mathcal{A}$ as initial condition, we have that $R(t)$ satisfies the system
\begin{align*}
\left\lbrace \begin{array}{ccc}
R'(t) & = & \dfrac{1}{R(t)}, \\
R(0) & = & r_0, \end{array} \right.
\end{align*}
so $R(t)=\sqrt{r_0^2+2t}$.

\noindent On the other hand, if we assume that $\psi\equiv0$ and that the area enclosed by the curve $\gamma_0$ is $Area(S_{r_0})/2$ then, $T=\infty$. 

\noindent Now, we can calculate the value of $\widetilde{T}$:
\begin{align*}
\widetilde{T}=\int_0^\infty\Big(\dfrac{r_0}{R(t)}\Big)^2\:dt=\int_0^\infty\dfrac{r_0^2}{r_0^2+2t}\:dt
=\dfrac{r_0^2}{2}\ln(r_0^2+2t)\big\vert_{0}^{\infty}=\infty.
\end{align*}
Then, the flow topologically subconverges to $\gamma_\infty:\mathbb{S}^1\rightarrow [0,\:\infty]\times\mathbb{S}^2,\:p\mapsto (\infty,\:\chi(p))$ with $\chi:\mathbb{S}^1\rightarrow S_{r_0}$ a smooth embedded $\psi$-minimal curve in $(S_{r_0},\:g_{S_{r_0}},\:\psi)$.
\end{remark}

In the following result, we provide an equivalence to the hypothesis about the function $B$ in the cases ii.a) and ii.b) of the last theorem. 

Let's consider $(\mathbb{R}^2,\:g_{w},\:\varphi)$ a Riemannian manifold with density where $g_{w}:=dr^2+w^2(r) g_{\mathbb{S}^1}$ and $\varphi=\varphi(r)$ then, $B(r)$ is the geodesic curvature with density of the $C_r$ circle centered at the origin whose radius is $r$ respect to the normal field $-\partial_r$.

\begin{proposition}
Given
\begin{align*}
L:[0,\infty)&\longrightarrow\mathbb{R}
\nonumber\\
r&\longmapsto L(r):=\ln\big(Length_\varphi(C_r)\big),
\end{align*}
where $\text{Length}_\varphi(C_r)$ is the length with density of the circle $C_r$ respect to the manifold $(\mathbb{R}^2,\:g_{w},\:\varphi)$. If we are in the situation of the case ii) of the last theorem, then:
\begin{itemize}
\item[ii.a)] The function $L\vert_{[r_0,\:\infty)}$  is Lipschitz if and only if $\liminf_{r\rightarrow\infty} B(r)$ is finite.
\item[ii.b)] The function $\Big(L\vert_{[r_0,\:\infty)}\Big)^{-1}$ is Lipschitz if and only if $\limsup_{r\rightarrow\infty} B(r)\neq 0$.
\end{itemize}
\end{proposition}
\begin{proof}
We note that $L\vert_{[r_0,\:\infty)}$ is a $C^1$-function.
\begin{itemize}
\item[Case ii.a)] 
If $L\vert_{[r_0,\:\infty)}$ is Lipschitz then its derivative is bounded in $[r_0,\:\infty)$, that is, there exists a $C>0$ such that $\vert L\vert_{[r_0,\:\infty)}'(r)\vert\leq C$ for all $r\in[r_0,\:\infty)$. As $ L\vert_{[r_0,\:\infty)}'(r)=B(r)$ we obtain that $\liminf_{r\rightarrow\infty} B(r)$ is finite.

Reciprocally, if we assume that the $\liminf_{r\rightarrow\infty} B(r)$ is finite as $B(r)<0\:\text{for\:all\:}r\in[r_0,\:\infty)$ then exists $C>0$ such that $\vert B(r)\vert<C\:\text{for\:all\:}r\in[r_0,\:\infty)$ and $L\vert_{[r_0,\:\infty)}'(r)=B(r)$ so $\vert L\vert_{[r_0,\:\infty)}'(r)\vert<C\:\text{for\:all\:}r\in[r_0,\:\infty)$. Therefore, $L\vert_{[r_0,\:\infty)}$ is Lipschitz.

\item[Case ii.b)] In the case ii) we have that 
\begin{align*}
L\vert_{[r_0,\:\infty)}'(r)=B(r)<0,\:\text{for\:all\:}r\in[r_0,\:\infty), 
\end{align*}
so $L\vert_{[r_0,\:\infty)}$ has inverse function from the inverse function theorem, moreover  $\Big(L\vert_{[r_0,\:\infty)}\Big)^{-1}(r)=\dfrac{1}{L\vert_{[r_0,\:\infty)}(r)}=\dfrac{1}{B(r)}$. From this equality we obtain the result, the proof is analogous to the case ii.a).
\end{itemize}
\end{proof}

\section{Gaussian density}

In this section we aim to study a particular case where the case ii.b) of Theorem \ref{TInComAsin} applies. We shall assume that $\psi\equiv 0$, in order to use the area variation formula enclosed by the curve, which ultimately provides us a feasible way to explicitly perform the calculus.

Let $(\mathbb{R}^3,\:g_{\mathbb{R}^3},\:\xi)$ be the 3-dimensional Euclidean space with a density $\xi$ such that $\psi(r)\equiv 0$ and $\varphi(r)=-\dfrac{1}{2}\mu^2 r^2$, that is, the radial part of the density is the Gaussian density. In this situation, given $\gamma_0\in\mathcal{A}$ as initial condition of the problem (\ref{PVI}), the system of ODE for $R(t)$ is:
\begin{align}
\left\lbrace \begin{array}{ccc}
R'(t) & = & -\dfrac{1}{R(t)}+\mu^2 R(t), \\
R(0) & =&r_0.\end{array} \right.
\end{align}
We could resolve this system:
\begin{align*}
\left.
\begin{array}{ll}
R'(t) =  -\dfrac{1}{R(t)}+\mu^2 R(t), &   R'(t) =  \dfrac{\mu^2R(t)^2-1}{R(t)},\\
\dfrac{R'(t)R(t)}{\mu^2R(t)^2-1}=1,& \ln\Big\vert\dfrac{\mu^2R(t)^2-1}{\mu^2R(0)^2-1}\Big\vert=2\mu^2 t,\\
\Big\vert\dfrac{\mu^2R(t)^2-1}{\mu^2 r_0^2-1}\Big\vert=e^{2\mu^2 t}, &  
\end{array} \right.
\end{align*}
and the solution is:
\begin{align}
 R(t)=\dfrac{1}{\mu}\sqrt{1+(\mu^2 r_0^2-1)e^{2\mu^2 t} }.
\end{align}

\begin{proposition}{\label{ProGaCamTi}}
The link between the time parameters is given by
\begin{align*}
\widetilde{t}:[0,\:T)&\longrightarrow [0,\:\widetilde{T})
\nonumber\\
t&\longmapsto \widetilde{t}(t)=\dfrac{r_0^2}{2} \ln\Big(\dfrac{\mu^2 r_0^2 e^{2\mu^2 t}}{1+(\mu^2 r_0^2-1)e^{2\mu^2 t}}\Big),
\\
t:[0,\:\widetilde{T})&\longrightarrow [0,\:T)
\nonumber\\
\widetilde{t}&\longmapsto t(\widetilde{t})=\dfrac{1}{2\mu^2}\ln\Big(\dfrac{ e^{2\widetilde{t}/r_0^2}}{\mu^2 r_0^2-(\mu^2 r_0^2-1)e^{2\widetilde{t}/r_0^2}}\Big).
\end{align*}
\end{proposition}
\begin{proof} From (\ref{CamTi}) the relation between times is given by
\begin{align*}
\widetilde{t}(t_1)&=\int_0^{t_1}\Big(\dfrac{w(r_0)}{w(R(t))}\Big)^2\:dt=\int_0^{t_1}\Big(\dfrac{r_0}{R(t)}\Big)^2\:dt
\\
&=r_0^2\int_0^{t_1} \dfrac{\mu^2}{1+(\mu^2 r_0^2-1)e^{2\mu^2 t} }\:dt
=r_0^2\int_1^{e^{2\mu^2 t_1}} \dfrac{\mu^2}{1+(\mu^2 r_0^2-1)x}\:\dfrac{1}{2\mu^2}\dfrac{dx}{x}
\\
&=\dfrac{r_0^2}{2}\int_1^{e^{2\mu^2 t_1}}\Big( \dfrac{1}{x}-\dfrac{(\mu^2 r_0^2-1)}{1+(\mu^2 r_0^2-1)x}\Big)\:dx
=\dfrac{r_0^2}{2}\Big( \ln(x)-\ln\big(1+(\mu^2 r_0^2-1)x\big)\Big\vert_1^{e^{2\mu^2 t_1}}
\\
&=\dfrac{r_0^2}{2} \ln\Big(\dfrac{x}{1+(\mu^2 r_0^2-1)x}\Big)\Big\vert_1^{e^{2\mu^2 t_1}}
\\
&=\dfrac{r_0^2}{2}\Big( \ln\Big(\dfrac{e^{2\mu^2 t_1}}{1+(\mu^2 r_0^2-1)e^{2\mu^2 t_1}}\Big)- \ln\Big(\dfrac{1}{1+\mu^2 r_0^2-1}\Big)\Big)
\\
&=\dfrac{r_0^2}{2} \ln\Big(\dfrac{\mu^2 r_0^2 e^{2\mu^2 t_1}}{1+(\mu^2 r_0^2-1)e^{2\mu^2 t_1}}\Big),
\end{align*}
where we have used the change of parameter $t=\dfrac{1}{2\mu^2}\ln x,\:dt=\dfrac{1}{2\mu^2}\dfrac{dx}{x}$. Now, we move on calculating the inverse function:
\begin{align*}
\left.
\begin{array}{ll}
\widetilde{t}=\dfrac{r_0^2}{2} \ln\Big(\dfrac{\mu^2 r_0^2 e^{2\mu^2 t}}{1+(\mu^2 r_0^2-1)e^{2\mu^2 t}}\Big), & e^{2\widetilde{t}/r_0^2}=\dfrac{\mu^2 r_0^2 e^{2\mu^2 t}}{1+(\mu^2 r_0^2-1)e^{2\mu^2 t}},  \\
(1+(\mu^2 r_0^2-1)e^{2\mu^2 t})e^{2\widetilde{t}/r_0^2}=\mu^2 r_0^2 e^{2\mu^2 t}, & e^{2\widetilde{t}/r_0^2}=\Big(\mu^2 r_0^2-(\mu^2 r_0^2-1)e^{2\widetilde{t}/r_0^2} \Big)e^{2\mu^2 t},\\
\dfrac{e^{2\widetilde{t}/r_0^2}}{\mu^2 r_0^2-(\mu^2 r_0^2-1)e^{2\widetilde{t}/r_0^2}}=e^{2\mu^2 t}, &  t=\dfrac{1}{2\mu^2}\ln\Big(\dfrac{e^{2\widetilde{t}/r_0^2}}{\mu^2 r_0^2-(\mu^2 r_0^2-1)e^{2\widetilde{t}/r_0^2}}\Big).
\end{array} \right.
\end{align*}
\end{proof}

Given $\gamma_0\in\mathcal{A}$ we shall denote by $\Omega_0$ to the region enclosed by the curve $\widetilde\gamma_0$ in the sphere $S_{r_0}$ such that $\dfrac{Area(\Omega_0)}{Area(S_{r_0})}\leq 1/2$, we also take the inward-pointing normal $\widetilde\nu$ to $\partial \Omega_0$.

\begin{theorem}
Let $\gamma_0\in\mathcal{A}$ and let  $\gamma:\mathbb{S}^1\times[0,\:T)\rightarrow \mathbb{R}^3$ be the unique maximal solution of the initial value problem (\ref{PVI}) with $\gamma_0$ as initial condition, then:
\begin{itemize}
\item[i)] If $r_0>\dfrac{1}{\mu}$:
\begin{itemize}
\item If $\dfrac{1}{2}\Big(1-\Big(\dfrac{\mu^2 r_0^2-1}{\mu^2 r_0^2}\Big)^{1/2}\Big)<\dfrac{Area(\Omega_0)}{Area(S_{r_0})}\leq 1/2$ the flow topologically converges to $\gamma_\infty:\mathbb{S}^1\rightarrow [0,\:\infty]\times\mathbb{S}^2,\:p\mapsto (\infty,\:\widetilde{\gamma}(p,\:\widetilde{T}))$.
\item  If 
$
\dfrac{Area(\Omega_0)}{Area(S_{r_0})}=\dfrac{1}{2}\Big(1-\Big(\dfrac{\mu^2 r_0^2-1}{\mu^2 r_0^2}\Big)^{1/2}\Big)
$
the flow topologically converges to a point $p_\infty\in\mathbb{S}_\infty\equiv \lbrace\infty\rbrace\times\mathbb{S}^2\subset [0,\:\infty]\times\mathbb{S}^2$ in the infinite radius sphere.
\item  If 
$
\dfrac{Area(\Omega_0)}{Area(S_{r_0})}<\dfrac{1}{2}\Big(1-\Big(\dfrac{\mu^2 r_0^2-1}{\mu^2 r_0^2}\Big)^{1/2}\Big)
$
the flow collapses to a spherical round point  in the Euclidean space $\mathbb{R}^3$.
\end{itemize}
\item[ii)] If $r_0=\dfrac{1}{\mu}$:
\begin{itemize}
\item If $\dfrac{Area(\Omega_0)}{Area(S_{r_0})}=1/2$ the flow $C^\infty$-subconverges, after a reparametrization of the curves $\gamma(\cdot,\:t)$, to a closed geodesic in $(S_{r_0},\:g_{r_0})$.
\item If $\dfrac{Area(\Omega_0)}{Area(S_{r_0})}<1/2$ the flow collapses to a round point in $(S_{r_0},\:g_{r_0})$.
\end{itemize} 
\item[iii)] If $r_0<\dfrac{1}{\mu}$:
\begin{itemize}
\item If $\dfrac{Area(\Omega_0)}{Area(S_{r_0})}=1/2$ the flow collapses to the coordinate origin in the Euclidean space $\mathbb{R}^3$. A blow-up centered at the origin coordinate gives a limit flow by the curve shortening problem in $(S_{r_0},\:g_{S_{r_0}})$ that $C^\infty$-subconverges, after a reparametrization of the curves, to a closed geodesic.
\item If $\dfrac{Area(\Omega_0)}{Area(S_{r_0})}<1/2$ the flow collapses to a spherical round point in $\mathbb{R}^3-\lbrace 0\rbrace$.
\end{itemize} 
\end{itemize}
\end{theorem}
\begin{proof}
\begin{itemize}
\item[]
\item[Case i)]
We note that the problem (\ref{PVIesfera}) is the curve shortening problem. It is know that we could calculate the maximal time of the flow $\widetilde{\gamma}$ from the variation formula of enclosed area by the curve $\widetilde\gamma(\cdot,\:\widetilde{t})$.  This formula is given by
\begin{align*}
\dfrac{\partial}{\partial \widetilde t}A(\Omega_{\widetilde{t}})&=-\int_{\mathbb{S}^1}k_{\wgamma,\:S_{r_0}}\: d\widetilde{s}=-2\pi+\int_{\Omega_{\widetilde{t}}}K_{S_{r_0}}\:da_{S_{r_0}}
\\
&=-2\pi+\dfrac{1}{r_0^2}A(\Omega_{\widetilde{t}}),
\end{align*}
where the first equality is know and we have used the Gauss-Bonnet theorem in the second equality, so:
\begin{align*}
\left.
\begin{array}{ll}
\dfrac{\partial}{\partial \widetilde t}A(\Omega_{\widetilde{t}})=-2\pi+\dfrac{1}{r_0^2}A(\Omega_{\widetilde{t}}), & \dfrac{\dfrac{\partial}{\partial \widetilde t}A(\Omega_{\widetilde{t}})}{-2\pi r_0^2+A(\Omega_{\widetilde{t}})}=\dfrac{1}{r_0^2}, \\
\dfrac{\partial}{\partial \widetilde t} \ln\big\vert -2\pi r_0^2+A(\Omega_{\widetilde{t}})\big\vert=\dfrac{1}{r_0^2},& \ln\big\vert \dfrac{-2\pi r_0^2+A(\Omega_{\widetilde{t}})}{-2\pi r_0^2+A(\Omega_{0})}\big\vert=\dfrac{1}{r_0^2}\widetilde{t},\\
-2\pi r_0^2+A(\Omega_{\widetilde{t}})=(-2\pi r_0^2+A(\Omega_{0}))e^{\widetilde{t}/r_0^2}, &  
\end{array} \right.
\end{align*}
and we obtain the following expression for the enclosed area:
\begin{align}
A(\Omega_{\widetilde{t}})=2\pi r_0^2-(2\pi r_0^2-A(\Omega_{0}))e^{\widetilde{t}/r_0^2}.
\end{align}
So, if $A(\Omega_0)=2\pi r_0^2$ the maximal time of the solution $\widetilde{\gamma}$ is infinite, let us remind that if the maximal time is finite, the curve collapses to a point and, if $A(\Omega_0)<2\pi r_0^2$, the maximal time is finite and it is given by 
\begin{align}
\widetilde{T}_{\max}= r_0^2\ln\Big(\dfrac{2\pi r_0^2}{2\pi r_0^2-A(\Omega_{0})} \Big).
\end{align}

From the hypothesis $r_0>1/\mu$ and Proposition \ref{ProGaCamTi} we obtain that the function $\widetilde{t}(t)$ is defined on $[0,\:\infty)$ and that $\lim_{t\rightarrow\infty}\widetilde{t}(t)= \dfrac{r_0^2}{2}\ln\Big(\dfrac{\mu^2 r_0^2}{\mu^2 r_0^2-1}\Big)$. Then
\begin{itemize}
\item If $\widetilde{T}_{\max}< \dfrac{r_0^2}{2}\ln\Big(\dfrac{\mu^2 r_0^2}{\mu^2 r_0^2-1}\Big)$, the flow $\gamma$ collapses to a point in the Euclidean space $\mathbb{R}^3$, because there is $t^\star \in [0,\:\infty)$ such that $\widetilde{t}(t^\star)=\widetilde{T}_{\max}$. Then, by the relation between the flows, $t^\star=T_{\max}<\infty$.
\item  If $\widetilde{T}_{\max}=\dfrac{r_0^2}{2}\ln\Big(\dfrac{\mu^2 r_0^2}{\mu^2 r_0^2-1}\Big)$, the flow $\gamma$ topologically converges to a point $p_\infty$ in $\mathbb{S}_\infty$ the infinite radius sphere.
\item If $\widetilde{T}_{\max}>\dfrac{r_0^2}{2}\ln\Big(\dfrac{\mu^2 r_0^2}{\mu^2 r_0^2-1}\Big)$, the flow topologically converges to $\gamma_\infty:\mathbb{S}^1\rightarrow [0,\:\infty]\times\mathbb{S}^2,\:p\mapsto (\infty,\:\widetilde{\gamma}(p,\:\widetilde{T}))$, curve contained in $\mathbb{S}_\infty$ the infinite radius sphere.
\end{itemize}
We can translate these inequalities in the following sense:
\begin{align*}
\lim_{t\rightarrow\infty}\widetilde{t}(t)&<\:(=)\:(>)\:\widetilde{T}_{\max}
\\
\dfrac{r_0^2}{2}\ln\Big(\dfrac{\mu^2 r_0^2}{\mu^2 r_0^2-1}\Big)&<\:(=)\:(>)\:
r_0^2\ln\Big(\dfrac{2\pi r_0^2}{2\pi r_0^2-A(\Omega_{0})} \Big),
\\
\Big(\dfrac{\mu^2 r_0^2}{\mu^2 r_0^2-1}\Big)^{1/2}&<\:(=)\:(>)\:
\dfrac{2\pi r_0^2}{2\pi r_0^2-A(\Omega_{0})},
\\
2\pi r_0^2-A(\Omega_{0})&<\:(=)\:(>)\:
\dfrac{2\pi r_0^2}{\Big(\dfrac{\mu^2 r_0^2}{\mu^2 r_0^2-1}\Big)^{1/2}},
\\
A(\Omega_{0})&>\:(=)\:(<)\:
2\pi r_0^2-\dfrac{2\pi r_0^2}{\Big(\dfrac{\mu^2 r_0^2}{\mu^2 r_0^2-1}\Big)^{1/2}},
\\
A(\Omega_{0})&>\:(=)\:(<)\:
2\pi r_0^2\Big(1-\Big(\dfrac{\mu^2 r_0^2-1}{\mu^2 r_0^2}\Big)^{1/2}\Big),
\\
\dfrac{A(\Omega_{0})}{A(S_{r_0})}&>\:(=)\:(<)\:
\dfrac{1}{2}\Big(1-\Big(\dfrac{\mu^2 r_0^2-1}{\mu^2 r_0^2}\Big)^{1/2}\Big).
\end{align*}
At this point, if we write the previous classification in these terms, we obtain the statement of the theorem.

\item[Case ii)] In this situation, $R(t)$ is constant for all t then, this case is the classic curve shortening problem \cite{GaHa86,Ga84, Gra89,Ga90b}.

\item[Case iii)] We note that if $r_0<\dfrac{1}{\mu}$ then, $T$ is the maximal time of the solution and it is less than or equal to $\dfrac{1}{2\mu^2}\ln\Big(\dfrac{1}{1-\mu^2 r_0^2}\Big)$. This time is the first time such that $R(\dfrac{1}{2\mu^2}\ln\Big(\dfrac{1}{1-\mu^2 r_0^2}\Big))=0$, so the maximal time is finite. Then from Theorem \ref{TeTieFinCo}, the curve collapses to a point; moreover, from Theorem \ref{TeTieFinSin} we know how the singularities are. We also notice that in this situation:
\begin{align*} 
\text{the\:flow\:collapses\:to\:the\:pole\:o}\Leftrightarrow\:T=\dfrac{1}{2\mu^2}\ln\Big(\dfrac{1}{1-\mu^2 r_0^2}\Big)\Leftrightarrow\:\widetilde{T}=\infty,
\end{align*}
and by the variation formula we obtain that
\begin{align*}
\widetilde{T}=\infty\Leftrightarrow\:\dfrac{Area(\Omega_0)}{Area(S_{r_0})}=1/2
\end{align*}
\end{itemize}
\end{proof}

\section*{Acknowledgements}

I thank to Dr. Vicent Gimeno and Dr. Vicente Palmer for their support and useful comments throughout our fruitful seminars that encouraged the writing of the present manuscript.

\bibliographystyle{acm}
\bibliography{Bibliografia}

\noindent
Department of Mathematics, School of Technology and Experimental Sciences, Universitat Jaume I, Av. de Vicente Sos Baynat, s/n, E-12071 Castelló de la Plana, Spain\\
{fvinado@uji.es}
\end{document}